\documentclass[11pt]{article}
\usepackage{geometry}
\usepackage[parfill]{parskip}

\usepackage{graphicx}

\usepackage{amssymb}

\usepackage{epstopdf}

\usepackage{amsmath,amsthm,amscd,amssymb}
\usepackage{latexsym}
\numberwithin{equation}{section}

\theoremstyle{plain}
\newtheorem{theorem}{Theorem}[section]
\newtheorem{lemma}[theorem]{Lemma}
\newtheorem{corollary}[theorem]{Corollary}
\newtheorem{proposition}[theorem]{Proposition}
\newtheorem{conjecture}[theorem]{Conjecture}

\theoremstyle{definition}
\newtheorem{definition}[theorem]{Definition}

\theoremstyle{remark}

\newtheorem{case[theorem]}{Case}

\newcommand{\fn}{F _3 ^N}
\newcommand{\beqa}{\begin{eqnarray*}}
\newcommand{\eeqa}{\end{eqnarray*}}
\newcommand{\beqan}{\begin{eqnarray}}
\newcommand{\eeqan}{\end{eqnarray}}
\newcommand{\one }{\mathbf 1 }

\newcommand{\Z}{\mathbb Z}

\title{New bounds on cap sets}
\author{Michael Bateman and Nets Hawk Katz}

\begin{document}

\maketitle

\begin{abstract}
We provide an improvement over Meshulam's bound on cap sets in $F_3^N$. We show that there exist universal $\epsilon>0$
and $C>0$ so that any cap set in $F_3^N$ has size at most $C {3^N \over N^{1+\epsilon}}$. We do this by obtaining
quite strong information about the additive combinatorial properties of the large spectrum. 

\end{abstract}


\section{Introduction}

A set $A \subset F_3^N$ is called a cap set if it contains no lines. In this paper, we will
be concerned with proving the following theorem:

\begin{theorem} \label{main}
There exists an $\epsilon  >0$, and $C < \infty$ such that if $A\subseteq \fn$ is a cap set, then 
\beqa
{{|A|} \over {3^N} } \leq {C \over {N^{1+ \epsilon  } } }.
\eeqa
\end{theorem}

The problem of the maximal size of cap sets is a characteristic 3 model for the problem of finding arithmetic
progressions of length 3 in rather dense sets of integers. Meshulam \cite{M95}  , through a direct use of ideas of Roth,
was able to show that there is a constant $C$ so that any cap set $A$ has density at most ${C \over N}$. Our
result may be viewed as a very modest improvement over Meshulam's result.

Sanders \cite{S11} recently showed that any subset of the integers whose density in $\{1,\dots,M\}$ is at least
${C (\log \log M)^5 \over \log M}$ must contain arithmetic progressions of length 3. This may be thought
of as bringing the results known for arithmetic progressions almost to the level of Meshulam's result. This has
spurred further interest in improving Meshulam's result in hopes that it might suggest a way of improving the
results on arithmetic progressions.

A rather concrete, though perhaps still out of reach, goal in this direction is a conjecture of Erd\"{o}s and Turan:

\begin{conjecture}
Suppose $A\subseteq \Z $ is such that 
\beqa
\sum _{n \in A } {1\over n} = \infty .
\eeqa
Then $A$ contains an arithmetic progression of every length.  
\end{conjecture}
It is clear that the present paper is directly relevant only for finding $3$-term progressions.  However it is 
also easy to see, based purely on density considerations, that proving an estimate of the type in Theorem \ref{main} in 
the integer setting would yield the $3$-term case of the conjecture stated above.  In fact, Polymath 6 \cite{PM6} has 
recently been started with the goal of adapting the ideas of this paper to the integer setting.  See \cite{PM} for more 
information about so-called ``polymath" projects in general.

While the research in this paper was well underway, Gowers \cite{G} wrote a post on his blog suggesting that one could 
attack
the problem of bounding cap sets by studying the additive structure of their large spectrum. This had been
our approach as well and we wrote a reply \cite{K1} sketching our rather strong results regarding that structure. In the
course of a few days, we realized that we actually could convert our structural theory into an estimate on the size of 
cap sets. We recorded this \cite{K2} in a second reply to Gowers's blog. The current paper should be viewed as an 
elaboration of these two posts.

We describe our plan for proving Theorem \ref{main}. To prove this theorem we will prove a theorem about sets 
without unusually dense subspaces, a notion we make precise below.
\begin{definition}\label{incrementdef}
We say a set $A$ has \it no strong increments \rm if for every subspace 
$V \subseteq \fn $ with $d=$ codim$V \leq {N \over 2}$, we have 
\beqa
{{|A \cap V| \over {|V|} } \leq \rho + {{20 d \rho }\over N} }.
\eeqa
\end{definition}

\begin{theorem} \label{noincrements}
There exists an $\epsilon >0$, and $C < \infty$ such that if $A\subseteq \fn$ is a cap set with no strong increments, 
then 
\beqa
{{|A|} \over {3^N} } \leq {C \over {N^{1+ \epsilon} } }.
\eeqa
\end{theorem}
Major ingredients needed to prove this theorem are Proposition \ref{first}, Lemma \ref{smoothing}, and 
Theorem \ref{specstruct}.  We combine them with a Fourier analytic argument in Section \ref{lastsection}.

\begin{proof}
We deduce Theorem \ref{main} from Theorem \ref{noincrements} using induction.  Suppose that for every $n\leq N-1$ we have
shown that if 
$B\subseteq F_3 ^n$ is a cap set, then 
\beqa
{{|B|} \over {3^n} } \leq {C \over {n ^{1+\epsilon}} }.
\eeqa
We aim for a contradiction:  assume there exists a cap set $A \subseteq \fn $ such that 
${{|A|} \over {3^N} } > {C \over {N^{1+ \epsilon  } } }$.  By Theorem \ref{noincrements}, this implies $A$ has a strong 
increment.    
Since $A$ has a strong increment, there exists an affine subspace $V \subseteq \fn$ with codimension $\leq {N \over 2}$ such that 
\beqa
{{ |A \cap V | } \over {|V|}} & \geq & \rho + {{20 d \rho }\over N } \\
& = & \rho (1 + { {20 d } \over N} ) \\
& > &  {C \over {(N-d)^{1+ \epsilon  } } }
\eeqa
since the derivative of ${{C}\over {x ^{1+\epsilon }} }$ is uniformly bounded by $16 C N^{-2-\epsilon} = 16 \rho N^{-1} $ whenever 
$0 < \epsilon < 1 $.  But we know that $A\cap V$ (in fact any subset of A) is a capset.  
This contradicts the induction hypothesis, yielding Theorem \ref{main}.
\end{proof}

\subsection{Proof Sketch}

We sketch our plan for proving Theorem \ref{noincrements}.  

We will study the large spectrum $\Delta$ of a cap set $A$ with no strong increments. The reader should
think of $\Delta$, for a cap set of the size ${3^N \over N}$  given by Meshulam's estimate, as consisting
of the positions at which the absolute value of the Fourier transform of $A$ is around ${1 \over N^2}$. The set
$\Delta$ should have cardinality approximately $N^3$ (established in Section \ref{reviewsection}) and have about 
$N^7$ additive quadruples (established in Section \ref{energy}).
Recall that an additive quadruple is a quadruple $(x_1,x_2,x_3,x_4)$ of elements of $\Delta$ with the
property that 
$$x_1+x_2=x_3+x_4.$$
Similarly an additive octuple is an octuple $(x_1,x_2,x_3,x_4,x_5,x_6,x_7,x_8)$ of elements of $\Delta$ with
$$x_1+x_2+x_3+x_4 =x_5+x_6+x_7+x_8.$$

It is easy to see {\it a priori} that a set with many additive quadruples will have many additive octuples.
In our case, a set with size $N^3$, having $N^7$ quadruples must have at least $N^{15}$ octuples. (But it
may have more octuples.) The number of octuples it has should be taken as an indication of its structure.
If there are many octuples, it means that the sumset $\Delta + \Delta$ looks like it has
more additive structure than the set $\Delta$. We then say the set $\Delta$ is additively smoothing. (It becomes
smoother under addition.) We show, however, that this cannot be the case for $\Delta$, the large spectrum.
We use the probabilistic method to do this, finding too much of the spectrum contained in a small
subspace in the additively smoothing case. We establish this in Section \ref{randomsection}.  
(This is somewhat reminiscent of the paper of Croot and Sisask \cite{CS11}
where random selections are used to uncover structure.)  

Thus our set $\Delta$ is entirely additively non-smoothing. 
This means it is already as smooth as it will
become under a small number of additions. This makes its additive structure particularly easy to uncover
as it is already present without adding the set to itself.  This kind of idea was first exploited
in a paper of the second author with Koester \cite{KK10} , and we use techniques quite similar to those found in that paper.
We end up showing that the set $\Delta$ should be thought of as looking like the sum of a very structured
set $K$ of size $N$ ( that is to say that $K$ is almost additively closed)  with a very random set $\Lambda$
of size $N^2$.  Section \ref{structuresection} contains the proof of a structural theorem for sets with substantial 
additive energy 
(i.e., many additive quadruples) but no additive smoothing.

We find that this structure of $\Delta$ is inconsistent with $A$'s being a cap set with no strong increments.
The reason is that we can use Freiman's theorem to place $K$ inside a subspace $H$ with relatively low dimension.
We can essentially mod out by $H$, examining the ``fibers", the intersections of $A$ with translates of
$H^{\perp}$. We find that the structure of $\Delta$ makes the behavior of the fibers unrealistic. This
argument is suggested by a paper of Sanders \cite{S10}, and is carried out in Section \ref{lastsection}.

One final remark about the value of $\epsilon$ obtained in this paper:  it is necessarily rather small 
(at least with our current argument).  We discuss the reasons for this in a brief final section and give some conjectures that, if true, would greatly improve the efficiency of our argument.
we have not attempted to optimize $\epsilon$ (or even keep track of exact dependence on $\epsilon$)
throughout the paper.  

{\bf Acknowledgements}  The first author is supported by an NSF postdoctoral fellowship, DMS-0902490.  The second author
is partially supported by NSF grant DMS-1001607.

The authors also thank Izabella \L aba and Olof Sisask for pointing out an error in the statement of the asymmetric Balog-Szemeredi-Gowers theorem in a preliminary version of this paper.


\section{Preliminaries} \label{additivelyclosed}

In this section, we record a few general results of which we will make frequent use throughout the paper.
We begin with our favorite form of the Cauchy-Schwarz inequality.

\begin{lemma} \label{cs}  Let $S$ and $T$ be finite sets and $\rho$ a map
$$\rho: S \longrightarrow T.$$
Let $P$ be the set of pairs
$$P=\{ (s_1,s_2): \rho(s_1)=\rho(s_2) \}.$$
Then
$$|P| \geq {|S|^2 \over |T|}.$$
\end{lemma}

\begin{proof} Note that we can express $|P|$ by
$$|P| = \sum_{t \in T}  |\rho^{-1}(t)|^2.$$
Applying the Cauchy-Schwarz inequality we see
$$|P| \geq  {1 \over |T|} (\sum_{t \in T} |\rho^{-1}(t)|)^2 = {|S|^2 \over |T|}.$$
\end{proof}

Here we introduce another variant of Cauchy Schwarz:

\begin{lemma} \label{Carbery} Let $(X,m)$ be a measure space with total measure $M$. Let $A_1,\dots A_k$ be
measurable subsets of $X$ and $0<\rho<1$ be a number (the density), so that $m(A_j) \gtrsim \rho M$ for each $j$.
Suppose $k >> {1 \over \rho}$. Then
$$\sum_{j=1}^k \sum_{l \neq j} m(A_j \cap A_l) \gtrsim k^2 \rho^2 M.$$
\end{lemma}

\begin{proof} Note that
$$\sum_{j=1}^k m(A_j) = \rho k M << k^2 \rho^2 M,$$
since $k \rho >> 1$.  Thus we may estimate the full sum
$$\sum_{j=1}^k \sum_{l=1}^k m (A_j \cap A_l).$$
Define $c(x)$ to be the measurable function giving for each $x$, the number of sets $A_j$ which contain $x$; i.e., 
\beqa
c(x) = \sum_{j=1} ^k \one _{A_j } (x).
\eeqa
Thus we would like to estimate
$$\int c(x)^2  \geq {1 \over M} (\int c(x))^2 = k^2 \rho^2 M.$$
\end{proof}

In what follows $\mu$ shall be a small exponent. We will frequently use expressions like
$N^{O(\mu)}$. The exponent will be bounded by $C \mu$ for $C$ a universal constant which varies from
line to line in the paper. We illustrate this by the following version of the large families principle
(this is the principle which says that most children belong to large families)
which will be used extremely often in this paper.

\begin{lemma} \label{largefamilies} Let $M_1, \dots M_K>0$ be real numbers and let $R>0$ be a real number.
Suppose that $M_j \leq R N^{\mu}$ for each $j$ and suppose that
$$\sum_{j=1}^K  M_j \geq  RK N^{-\mu}.$$
Then there exists a subset $J$ of $\{1, \dots , K\}$ with $|J| \gtrsim N^{-O(\mu)} K$ so that
for each $j \in J$, we have $M_j \gtrsim N^{-O(\mu)} R.$
\end{lemma}

\begin{proof} Let 
$$J=\{j: M_j \gtrsim N^{-10 \mu} R \}.$$
Suppose that $|J| \lesssim N^{-10 \mu} K$. Then by the upper bound on $M_j$, we have that
$$\sum_{j \in J} M_j \lesssim N^{-9 \mu} R K,$$
while
$$\sum_{j \notin J} M_j \lesssim N^{-10 \mu} R K.$$
Combining these two estimates gives us a contradiction.
\end{proof}

We take a moment to state the asymmetric Balog-Szemeredi-Gowers theorem which we will have occasion to use.
A set $B \subset F_3^N$ will be said to be $\mu$-additively closed if
$$|B-B| \lesssim N^{O(\mu)}  |B|.$$

\begin{lemma} \label{absg} Let $B,C \subset F_3^N$ so that there are at least $N^{-\eta} |B| |C|^2$ additive
quadruples of the form
$$b_1+c_1=b_2+c_2$$
with $b_1,b_2 \in B$ and $c_1,c_2 \in C$. Let $L = { |B| \over |C|}$, and assume $L \lesssim N^{10}$.  Then there 
exists $\mu$ depending only on $\eta$, with $\mu \rightarrow 0$ as $\eta \rightarrow 0$, and there exist 
an $\mu$-additively closed set $K \subset F_3^N$ and a set
$X \subset F_3^N$ with 
$$|X| \lesssim  N^{O(\mu)} {|B| \over |C|},$$
so that
$$|B \cap (X+K)| \gtrsim N^{-O(\mu)} |B|,$$
and an element $x \in F_3^N$ so that
$$|C \cap (x+K)| \gtrsim N^{-O(\mu)} |C|.$$
In particular, the last inequality implies
$$|K| \gtrsim N^{-O(\mu)} |C|.$$
\end{lemma}

An only slightly stronger form of this lemma appears in the book of Tao and Vu as Lemma 2.35.  \cite{TV06}
Another way of stating this result which we will use frequently is to define a function $f$ with
$\lim_{t \longrightarrow 0} f(t)=0$ and to let $\mu=f(\eta)$. We will use this kind of notation
frequently in the paper with the choice of $f$ varying from line to line.

Finally, we record the form of Freiman's theorem which we shall use.

\begin{theorem} \label{Freiman} A $\mu$ additively closed set is contained in a subspace of dimension
$N^{O(\mu)}$. \end{theorem}

Various improvements of the finite characteristic Freiman's theorem have occured  such as the
result by Sanders \cite{S08} but these only affect the constant in our formulation. Even Ruzsa's original version 
\cite{R99} suffices.


\section{Review of Meshulam's argument, Fourier analysis in $F_3^N$, and sparsity of the spectrum} \label{reviewsection}

For the remainder of this paper $A$ will be a subset of $F_3^N$ with $|A| = \rho 3^N >> 
{3^N \over N^{1+\epsilon}}$
with $\epsilon>0$ to be determined later. Moreover $A$ shall be a cap set meaning that it 
contains no
lines. A line in $F_3^N$ is characterized by being a set with exactly three distinct 
elements $a,b,c$
satisfying $a+b+c=0$. 

In this section we will establish some basic facts needed for our proof, and that are enough to obtain 
Meshulam's bound of $\sim {1\over N}$  on the density of capsets.  Further, we shall
prove a statement of the form ``The spectrum $\Delta $ does not have too much intersection with 
any small subspace."

We will assume that there are no strong increments for $A$ in the sense of Definition \ref {incrementdef}. 
Precisely, we assume there is 
no hyperplane $H$ so that
$A \cap H$ has density $\geq \rho + {20 \rho \over N}$ in $H$ and no subspace $H$ of 
codimension $d \leq  {N \over 2}$
so that $A \cap H$ has density $\geq  \rho + {20 \rho d \over N}$.  We recall that a contradiction of 
this assumption will mean
that every large cap set $A$ has strong increments. This will contradict the existence of 
large cap sets.

We define the character $e: F_3 \longrightarrow {\bf C}$ by
$e(0)=1,e(1) =-{1 \over 2} +{\sqrt{3} \over 2} i, e(2)=-{1 \over 2} - {\sqrt{3} \over 2} 
i.$
We will study the Fourier transform of the set $A$, namely
$$\hat A(x) = {1 \over 3^N} \sum_{a \in A}  e(a \cdot x).$$
As a consequence of the assumption that $A$ is a large cap set without strong increments,
we shall see that there is a significant set $\Delta$ of $x$ for which $|\hat A(x)|$ is fairly 
large (the set $\Delta$ will
be called the spectrum of $A$) and we shall see that $\Delta$ does not concentrate too 
much in any fairly low
dimensional subspace of $F_3^N$. 
Our first nontrivial fact about $\hat{A}$ is that $\widehat{A - \rho }$ has large $L^3$ norm, and moreover that this 
large $L^3$ norm is accounted for by the set of $x$ where $|\hat{A}(x)|$ is large.  Precisely: 
\begin{definition}\label{spectrumdef}
Define 
\beqa
\Delta = \{ x \neq 0 \colon | \hat{A} (x) | \gtrsim \rho ^2 \}.
\eeqa
We shall refer to the set $\Delta$ as the \it spectrum \rm of $A$ and it shall be our central object of study for the
remainder of the paper. 
\end{definition}
Note that with this definition, $\Delta$ is symmetric, that is
$$\Delta=-\Delta.$$

It is worth noting that the following proposition is the only place in the paper where we use the assumption 
that $A$ is a capset specifically; the other parts of the paper use only the assumption that $A$ has no strong increments.
\begin{proposition} \label{size}
If $A$ is a capset, then 
\beqa
\sum _{x \neq 0} |\hat{A}(x) |^3 \gtrsim \rho ^3. 
\eeqa
Moreover,
\beqa
\sum _{x \in \Delta } |\hat{A}(x) |^3 \gtrsim \rho ^3,
\eeqa
and 
\beqa
N^3 \lesssim |\Delta | \lesssim N^{3 + 3\epsilon }.
\eeqa
\end{proposition}
Note that this proposition is already enough to obtain Meshulam's estimate:  in particular, it guarantees that the 
set $\Delta$ defined in the statement of the proposition is nonempty.  This means there is at least one $x$ such 
that $|\widehat{A} (x) | \gtrsim \rho ^2 $.  This guarantees the existence of a hyperplane $P$ such that the
density of $A\cap P$ inside $P$ is at least $ \rho + c\rho ^2 $.  Taking $\rho$ large enough compared to ${1\over N}$ 
already contradicts the no-strong-increment hypothesis, yielding Meshulam's estimate.  

In what follows, we prove this proposition.  We consider
$$\sum_{x \in F_3^N} \hat  A(x)^3 = {1 \over 3^{3N}} 
\sum_{x \in F_3^N} \sum_{a \in A} \sum_{b \in A} \sum_{c \in A} e((a+b+c) \cdot x).$$
Summing first in $x$, we see that this expression yields $3^{-2N}$ mutiplied by the number of solutions of the
equation $a+b+c=0$ with $a,b,c$ taken from $A$. Since we have assumed that $A$ is a cap set, the only solutions
occur when $a=b=c$.  Thus
$$\sum_{x \in F_3^N} \hat A(x)^3= 3^{-2N} |A|= 3^{-N} \rho.$$
However, we observe that $\hat A(0)= \rho$. Thus, given the size of $A$, we see that $\rho^3$ dominates $3^{-N} \rho$ and
we conclude
\begin{equation} \label{L3} \sum_{x \neq 0} |\hat A(x)|^3 \gtrsim  \rho^3.\end{equation}

However, following the proof of Plancherel's inequality, we see that
$$\sum_{x \neq 0} |\hat A(x)|^2 \leq \sum_{x \in F_3^N} |\hat A(x)|^2 = 
{1 \over 3^{2N}} \sum_{x \in F_3^N} \sum_{a \in A} \sum_{b \in B} 
e((a-b) \cdot x).$$
Summing first in $x$, we conclude
\begin{equation} \label{L2}\sum_{x \neq 0} |\hat A(x)|^2 \leq  3^{-N} |A| \leq  \rho. \end{equation}

By the assumption that $A$ has no strong codimension 1 increment, we conclude
\begin{equation} \label{inc1} |\hat A(x)| \lesssim {|A| \over N 3^N}  = {\rho  \over N}. \end{equation}
Recall that 
$$\Delta =\{ x \neq 0 : | \hat A(x)| \gtrsim  \rho^2 \}.$$
By selecting the implicit constant in the definition of $\Delta$ correctly, we see combining inequalities \ref{L3}
and \ref{L2} that
\begin{equation} \label{L3spec} \sum_{x \in \Delta} |\hat A(x)|^3 \gtrsim \rho^3 . \end{equation}

Combining inequalities \ref{inc1} and \ref{L3spec}, we see that
$$|\Delta| \gtrsim N^3.$$

Combining the definition of $\Delta$ with \ref{L2} we get
$$|\Delta| \lesssim N^{3+3 \epsilon}.$$

Now we prove a statement of the form ``The spectrum $\Delta $ does not have too much intersection with 
any small subspace."  Precisely:

\begin{proposition} \label{first} Let $A$ be a set without strong increments. 
Let $\Delta$ be the spectrum, as in Definition \ref{spectrumdef}.
Then for any subspace $W$ of $F_3^N$ having dimension $d \leq {N \over 2}$, we have
$$|\Delta \cap W| \lesssim  d N^{1 + 2 \epsilon}.$$
Moreover for such a subspace $W$, we have the estimate 
$$ \sum_{w \neq 0  \in W} |\hat A(w)|^2 \lesssim  \rho^2 {d \over N} .$$
\end{proposition}

Here we see what the assumption of no higher codimension strong increments implies about the spectrum $\Delta$. Let
$H$ be a subspace with codimension $d<{N \over 2}$, we let $V$ be a dimension $d$ subspace which is
transverse to $H$ (i.e., $V+ H = \fn $)
and we let $W$ be the annihilator space of $H$. Then for any $w \neq 0 \in W$, we see that
$$\hat A(w)={1 \over 3^N}\sum_{v \in V}  ( |A \cap (H+v)| - 3^{-d} |A|) e(v \cdot w).$$

Then we have
\begin{equation} \label{Planch} \sum_{w \neq 0  \in W}
|\hat A(w)|^2 = 3^{d-2N}  \sum_{v \in V} ( |A \cap (H+v)| - 3^{-d} |A|)^2. \end{equation}
By getting an upper bound on the right hand side of equation \ref{Planch}, we can obtain an upper bound on 
$|\Delta \cap W|$, which is our goal.

To estimate the right hand side, we subdivide $V=V^+ \cup V^{-}$, where
$$V^+ = \{v : |A \cap (H+v)| - 3^{-d} |A| \geq 0\}$$
and
$$V^{-} = \{v : |A \cap (H+v)| - 3^{-d} |A| < 0\}$$

We observe that since 
$$\sum_{v \in V}   (|A \cap (H+v)| - 3^{-d} |A|) = 0,$$
we have
\begin{equation} \label{equal} 
\sum_{v \in V^+} | |A \cap (H+v)| - 3^{-d} |A| | = \sum_{v \in V^{-}} | |A \cap (H+v)| - 3^{-d} |A| |. \end{equation}

We observe that the hypothesis that $A$ has no strong increments implies that for $v \in V^+$, we have the
estimate
$$ | |A \cap (H+v)| - 3^{-d} |A| | \lesssim 3^{-d} |A|  {d \over N}.$$

Thus simply using that $|V|=3^d$, we get the estimates

\begin{equation} \label{L1+} \sum_{v \in V^+} | |A \cap (H+v)| - 3^{-d} |A| | \lesssim |A| {d \over N} \end{equation}

and

\begin{equation} \label{L2+} \sum_{v \in V^+} | |A \cap (H+v)| - 3^{-d} |A| |^2 \lesssim 3^{-d} |A|^2 {d^2 \over N^2} 
\end{equation}

Now for $v \in V^{-}$, we have the trivial estimate
$$ | |A \cap (H+v)| - 3^{-d} |A| | \leq 3^{-d} |A|.$$

In light of equation \ref{equal} and estimate \ref{L1+} this yields

\begin{equation} \label{L2-} \sum_{v \in V^{-}} | |A \cap (H+v)| - 3^{-d} |A| |^2 \lesssim 3^{-d} |A|^2 {d \over N} 
\end{equation}

Combining inequalities \ref{L2+} and \ref{L2-}, gives the estimate
$$ \sum_{v \in V} | |A \cap (H+v)| - 3^{-d} |A| |^2 \lesssim 3^{-d} |A|^2 {d \over N} .$$

Thus equation \ref{Planch} gives

$$ \sum_{w \neq 0  \in W} |\hat A(w)|^2 \lesssim  \rho^2 {d \over N} .$$

However, we recall that if $w \in \Delta$, we have that 
$$|\hat A(w)| \gtrsim \rho ^2 = { 1 \over N^{2+ 2\epsilon} }.$$

Thus we get the desired estimate:
$$|\Delta \cap W| \lesssim  d N^{1 + 2 \epsilon}.$$

\section{Additive structure in the spectrum of large cap sets} \label{energy}
In this section we establish that the spectrum has some nontrivial additive structure.  Specifically, we prove it has 
$N^{7 - O(\epsilon)}$ additive quadruples.

\begin{proposition} \label{partialstructure} Let $A$ be a large cap set. Let $\Delta$ be the spectrum of $A$
and let $\Delta^{\prime}$ be any symmetric subset of $\Delta$ with 
$$|\Delta^{\prime}| \geq {1 \over 2} |\Delta|.$$
Let $E_4(\Delta^{\prime})$ be the number of additive quadruplets
$x_1+x_2=x_3+x_4$ with $x_1,x_2,x_3,x_4 \in \Delta^{\prime}$. Then
$$E(\Delta^{\prime}) \gtrsim {|\Delta|^4 \over N^{5 + 5 \epsilon}}.$$
\end{proposition}
The argument for a major subset $\Delta '$ of $\Delta $ is no different, so for convenience of notation we assume in fact
$\Delta ^{\prime} = \Delta $.

We retain the notation of the previous section considering $\Delta$ the spectrum of a large cap set $A$.
In particular, we have $|A| >> {3^N \over N^{1+\epsilon}}$, we have
$$N^3 \lesssim |\Delta| \lesssim N^{3 + 3 \epsilon},$$
we have for every $x \in \Delta$, that $|\widehat{A} (x)| \gtrsim {|A| \over N^{1+\epsilon}}$, and we have
that $\Delta$ is symmetric, namely $\Delta=-\Delta$.

From the lower bound on $|\widehat{ A }(x)|$, we have for each $x$, an affine hyperplane $H_x$, annihilated by $x$
so that
$$|A \cap H_x| - {1 \over 3} |A| \gtrsim 3^N{ \rho \over N^{1 + \epsilon}} = { |A| \over  N^{1 + \epsilon} }.$$

Summing over $\Delta$, we obtain

$$\sum_{x \in \Delta} (|A \cap H_x| -{|A| \over 3} )\gtrsim {|A| |\Delta| \over N^{1 + \epsilon}}.$$

We wish to rewrite this as a double sum by introducing $1_{H_x}$, the indicator function of $H_x$.

$$\sum_{x \in \Delta} \sum_{y \in A} (1_{H_x}(y) - {1 \over 3}) \gtrsim {|A| |\Delta| \over N^{1 + \epsilon}}.$$

We interchange the order of the sums:

$$\sum_{y \in A} (\sum_{x \in \Delta} (1_{H_x}(y) - {1 \over 3})) \gtrsim {|A| |\Delta| \over N^{1 + \epsilon}}.$$

Now we apply H\"older's inequality:

$$|A|^{{3 \over 4}} ( \sum_{y \in A} |\sum_{x \in \Delta} (1_{H_x}(y) - {1 \over 3})|^4)^{{1 \over 4}}
\gtrsim {|A| |\Delta| \over N^{1 + \epsilon}}.$$

Taking everything to the fourth power and simplifying, we get

$$ \sum_{y \in A} |\sum_{x \in \Delta} (1_{H_x}(y) - {1 \over 3})|^4
\gtrsim {|A| |\Delta|^4 \over N^{4 + 4\epsilon}}.$$

Crudely expanding the sum, we get the apparently weaker inequality

$$ \sum_{y \in F_3^N} |\sum_{x \in \Delta} (1_{H_x}(y) - {1 \over 3})|^4
\gtrsim {|A| |\Delta|^4 \over N^{4 + 4\epsilon}}.$$

We can rewrite this as

\begin{equation}  \label{Batemanrulz} \sum_{y \in F_3^N} \sum_{x_1,x_2,x_3,x_4 \in F_3^N} 
\Pi_{\alpha=1}^4  (1_{H_{x_{\alpha}}}(y) -{1 \over 3}) \gtrsim {|A| |\Delta|^4 \over N^{4 + 4\epsilon}}.
\end{equation}

We claim the estimate \ref{Batemanrulz} says that the spectrum $\Delta$ has substantial additive structure. This will
be demonstrated by the following proposition.

\begin{proposition} \label{Batemanroks} Let $x_1,x_2,x_3,$ and $x_4$ be non-zero elements of $F_3^N$.
Then the expression
$$\sum_{y \in F_3^N}  \Pi_{\alpha=1}^4  (1_{H_{x_{\alpha}}}(y) -{1 \over 3}) \lesssim  3^N$$
Moreover it vanishes unless an equality of the form
$$\pm x_1 \pm x_2 \pm x_3 \pm x_4=0$$
holds.
\end{proposition}

\begin{proof}

We introduce the Fourier transforms of the balanced function of the hyperplanes $H_{x_{\alpha}}$ setting
$$f_{\alpha}(z)= {1 \over 3^N} \sum_{x \in F_3^N}  (1_{H_{x_{\alpha}}}(y) -{1 \over 3}) e(y \cdot z).$$
Then we use the standard Fourier identity
$$\sum_{y \in F_3^N}  \Pi_{\alpha=1}^4  (1_{H_{x_{\alpha}}}(y) -{1 \over 3})
=3^N \sum_{z_1+z_2+z_3+z_4=0}  f_1(z_1) f_2(z_2) f_3(z_3) f_4(z_4).$$
We observe that $f_j(z_j)$ vanishes unless $z_j = \pm x_j$. 

The upper bound on the sum just follows from the triangle inequality. 

\end{proof}

To finish the proof of Proposition \ref{partialstructure}, we apply the inequality \ref{Batemanrulz}, 
the fact that $|A| >> {3^N \over N^{1+\epsilon}}$, 
the proposition \ref{Batemanroks} and the fact that the spectrum $\Delta$ is symmetric. 

%

%
%
%
%


\section{Random selection argument for additively smoothing spectrum} \label{randomsection}

In this section we study the additive properties of random subsets of the spectrum.  We will show that they 
typically have very poor additive structure.  This will allow us to conclude that, although the spectrum has 
many $4$-tuples, it cannot have too many $8$-tuples.  The significance of this will only be made clear in 
Section \ref{structuresection}.

We defined $E_{4}(\Delta)$ to be the number of additive quadruplets in $\Delta$.  We define $E_{2m}(\Delta)$
to be the number of additive $2m$-tuples  
$$x_1+x_2 + \dots x_m = x_{m+1} + x_{m+2} + \dots x_{2m},$$
such that $x_1, x_2, \dots , x_{2m} \in \Delta $.  We let $\widehat{\Delta}(x)$ be the Fourier transform:
$$\widehat{\Delta}(y)={1 \over 3^N} \sum_{x \in \Delta} e(y \cdot x).$$
Then
$$E_{2m}(\Delta) = 3^{(2m-1)N} \sum_{y \in F_3^N} |\hat \Delta(y)|^{2m},$$
which can be sen by expanding the sum on the right, as in the proof of Plancherel's theorem.  
We always have $E_2(\Delta)=|\Delta|$. When we have nontrivial amounts of additive structure in the
sense that say $E_{2k}(\Delta) >> |\Delta|^k$, we can lift this up to counts of higher-tuplets using
H\"older's inequality. (We use the inequality to bound the $2k$-norm by the $2$-norm and the $2m$-norm.)
We can view this process as a poor man's Plunnecke theorem. We record this
result for high $E_4$ and high $E_8$.

\begin{lemma} \label{holder} Let $m> 2$. Then
$${E_4(\Delta)^{m-1} \over |\Delta|^{m-2} } \leq E_{2m} (\Delta).$$
Suppose $m>3$. Then
$${E_8(\Delta)^{m-1 \over 3}  \over |\Delta|^{m-4 \over 3}} \leq E_{2m} (\Delta).$$
\end{lemma}

\subsection{Discussion of additive smoothing}

We are now ready to introduce the notion of additive smoothing. We keep in mind two examples of kinds of sets
having additive structure. One kind of set consists of a subspace plus a random set. The other consists
of a random subset of a subspace. We think of the first kind of set as not being additively smoothing because
as you add it to itself, its expansion rate stays essentially constant. This is the kind of example for which
Lemma \ref{holder} is close to sharp. But the second kind of set, when added to itself, will quickly
fill out the subspace and its rate of additive expansion will shrink dramatically. The lack of additive
structure smooths out under addition. This is the kind of example for which Lemma \ref{holder} is far from
sharp.

We will momentarily define $\Delta$ to be additively smoothing if $E_8(\Delta)$ is substantially larger than
expected from Lemma \ref{holder} and our lower bound for $E_4(\Delta)$ obtained in Section \ref{energy}.
(Nonetheless, for the purposes of this paper, the gain in the exponent need 
only be $O(\epsilon)$.) We will define additive smoothing
so that if $\Delta$ is additively smoothing then for some not very large $m$ (depending only on $\epsilon$), we may expect to find
additive $m$-tuplets of $\Delta$ in a randomly chosen set $S$ of $d$ elements. 

Before we formally define the property of additive smoothing, we illustrate how the calculation works in the 
case $\epsilon=0$. In that case $d \sim N$ so that an element of $\Delta$ (which has size $N^3$) is chosen
with probability $N^{-2}$. We have a lower bound of $N^7$ on $E_4(\Delta)$. Suppose that we can improve
on the lower bound of $N^{15}$ for $E_8(\Delta)$ which we get from the first part of Lemma \ref{holder}
and in fact
$$E_8(\Delta) > N^{15+\delta}$$
for some $\delta>0$.
Then from the second part of Lemma \ref{holder}, we obtain the estimate
$$N^{(4+{\delta \over 3}) m + 7 -{\delta \over 3} } \leq E_{2m}(\Delta).$$
Thus there is some $m$ which depends only on $\delta$ so that
$$E_{2m}(\Delta) >> N^{4m}.$$
Thus the expected number of $2m$-tuplets in $S$ is $>> 1$. We will formally define
additive smoothing to achieve the same effect when $\epsilon$ is different from 0.

\subsection{Nonsmoothing of the spectrum}
In this subsection we make rigorous the arguments of the last subsection.
\begin{definition}
We define $\Delta$ to be additively smoothing if there is some $\sigma >  30 \epsilon$
so that $E_8(\Delta) >> N^{15 + \sigma}.$ 
\end{definition}
We are now in a position to state the main result of this section.

\begin{lemma} \label{smoothing} If $\Delta$ is the spectrum of a large cap set without strong increments
then $\Delta$ is not additively smoothing. \end{lemma}

We begin with a few comments about our proof strategy in this section.  If $S$ is a ``random" subset of $\Delta $, then we expect 
\beqa
E_{2m} (S) = \left({|S| \over |\Delta |}\right)^{2m} E_{2m} (\Delta).
\eeqa
Thus we can show that $E_{2m} (\Delta)$ is small by showing that $E _{S}$ is small for a typical 
(somewhat large) subset $S$ of $\Delta$.  


Now we fix a particular number $d$ and consider random subsets of $\Delta$ of size $d$.  We will take 
\begin{equation} \label{ssize}  d \sim N^{1-\epsilon} \end{equation}
with the explicit constant to be determined later.  Our first goal is to prove that we expect this subset to 
span a space of dimension $d$.  More precisely:  
\begin{definition}
Let $S$ be a set of $d$ vectors $x_1, \dots, x_d \in F_3^N$. We say that the set $S$ has nullity
$k$ if the dimension of the span of $S$ is $d-k$. 
\end{definition}
We will consider uniform random selections of sets of
$d$ elements from $\Delta$. We can view these selections as $d$-fold repetitions of uniform selection without
replacement.  We will prove

\begin{lemma} \label{random} A random selection $S$ of size $d$ from the spectrum $\Delta$ has nullity at
least $k$ with probability $\lesssim 2^{-k}$. \end{lemma}

\begin{proof} 
Once we have completed our first $m$ choices, our selections $x_1,\dots,x_m$ span a vector
space $W_m$ with dimension no more than $m$. Thus $|\Delta \cap W_m| \lesssim m N^{1+2 \epsilon}$
by Proposition \ref{first}. We choose
the constant in \ref{ssize} so that the probability that the $m+1$st element of $S$ lies in $W_m$
is bounded by ${1 \over d}$ for all $m \leq d-1$.  Note that since $m \leq d$, this probability is bounded by 
\beqa
{ {|\Delta \cap W_m| } \over {|\Delta | } } & \leq & {{ Cd N^{1+2 \epsilon}  } \over {N^3} }  \\
& = & {Cd \over {N^{2- 2 \epsilon} } } \\ 
& \leq & { 1 \over d } 
\eeqa 
provided $d << {N^{1- \epsilon}}$.
Thus the probability that $S$ has nullity at
least $k$ is bounded by the probability that for $d$ independent events with probability ${1 \over d}$
at least $k$ occur. 

The probability that exactly $k$ events from $d$ independent events with probability ${1 \over d}$
occur is exactly
$$g(k,d)={d \choose k} {(d-1)^{d-k} \over d^d}.$$
The numbers $g(k,d)$ decrease by a factor of more than 2 as $k$ is increased by 1 as long as $k>2$.  
This completes the proof of the lemma.
\end{proof}
Now that we know our random subset is likely to have full rank, we estimate the number of $2m$-tuples it contains in the case it 
does not have full rank.  Given a set $S$ with nullity $k$ we will bound the number of possible additive $2m$-tuplets
between elements of $S$.  Specifically:
\begin{lemma}
A set $S$ of size $d$ and nullity $k$ has $E_{2m} (S) \lesssim C_m k^{2m} $.
\end{lemma} 
\begin{proof}
We write a list $E$ of all equations among elements of $S$ which involve $2m$ or
fewer elements of $S$. Because the nullity is $k$, the span of these equations has dimension at most $k$.
We pick a basis $B$ for $E$ and the equations in $B$ involve at most $2mk$ elements of $S$. Thus all of the
equations of $E$ involve at most $2mk$ elements of $S$. Thus there are at most 
$$h(m,k)= 2^m (2m)! {2mk \choose 2m},$$
additive $2m$-tuplets from $S$. We refer to $h(m,k)$ as the number of possible $m$-tuplets in $S$. Note
that $h(m,k)$ is a polynomial of degree $2m$ in $k$.
\end{proof}
\begin{proof} [Proof of Lemma \ref{smoothing} ]

Now let $S$ be a random selection of $d$ elements from $\Delta$. Then by Lemma \ref{random}, the probability
that $S$ has nullity $k$ is $\lesssim 2^{-k}$. Thus the expected value of 
the number of possible $2m$-tuples $\sum _{k\geq 0} h(m,k)$ is $\lesssim_m  \sum _{k \geq 0} 2^{-k} k^{2m} \lesssim 1$.

Now we will show that we have defined additive smoothing so that the expected number of $2m$-tuples is $>>1$. 
This will give us a contradiction.

We know that $d \gtrsim N^{1-\epsilon}$. Thus our selection $S$ will be expected to have $>> 1$ non-trivial
$2m$-tuples, whenever $E_{2m}(\Delta) >> N^{4m+2m\epsilon}.$ (We simply calculate the probability that an individual
$2m$-tuple involves only elements of $S$.) Thus we may assume  that
$$E_{2m}(\Delta) \lesssim  N^{4m+8m\epsilon}.$$

Using the fact that $|\Delta| \lesssim N^{3+3\epsilon}$ and the second part of Lemma \ref{holder}, we get that
$$E_{8}(\Delta) \lesssim N^{{15m + 27m\epsilon -1 - \epsilon \over m-1}}.$$
Choosing $m$ sufficiently large gives
$$E_{8}(\Delta) \lesssim N^{15 + 27 \epsilon}.$$
The choice of $m$ and hence the constants depend on $\epsilon$ but not on $N$.

\end{proof}

\section{Structure of robust additively non-smoothing sets} \label{structuresection}

In this section, the only properties of the spectrum $\Delta$ which we shall use are its size, its additive structure, and its 
non-additive smoothing. Consequently the results can be stated in somewhat more generality. We leave intact,
however, the numerology coming from the case of spectrum of cap sets.

We will say that a symmetric set $\Delta \subset F_3^N$ is a robust additively non-smoothing set of strength $\delta$ provided that we know its size:
\begin{equation} 
N^3 \lesssim |\Delta| \lesssim N^{3+\delta}
\end{equation}
that we know how many additive quadruples can be made from any large subset of it, namely that if
$\Delta^{\prime} \subset \Delta$ with $|\Delta^{\prime}| \geq {3 \over 5} |\Delta|$ and
$\Delta^{\prime}$ symmetric, we have
\begin{equation} \label{robust}
E_4(\Delta^{\prime}) \gtrsim N^{7-\delta}
\end{equation}
and  that we have additive non-smoothing, namely
\begin{equation} \label{octuples}  
E_8(\Delta) \lesssim N^{15 + \delta}
\end{equation}
and moreover that for each element $a \in \Delta$, there are at most $N^{4+\delta}$ quadruples of the form
$$\pm a \pm b=\pm c \pm d$$
with $b,c,d \in \Delta$.

Let us pause to consider the case of $\Delta$, the spectrum of a cap set with no strong increments. We know
that the number of $a \in \Delta$ participating in more than $N^{4+O(\epsilon)}$ quadruplets
$$\pm a \pm b=\pm c \pm d$$
is smaller than ${1 \over 10} |\Delta|$ since otherwise $\Delta$ would have more quadruples and hence more
octuples than allowed by Lemma \ref{smoothing}. Let $\Delta^{\prime}$ be the remaining elements of $\Delta$.
Note that by its definition $\Delta^{\prime}$ is still symmetric. Note that any symmetric subset of $\Delta^{\prime}$
containing at least three fifths of its elements must contain at least half the elements of $\Delta$.
Thus from Proposition \ref{size}, Proposition \ref{partialstructure}, and Lemma \ref{smoothing} we know that:

\begin{proposition} \label{Spectrumsays} Let $\Delta$ be the spectrum of a large capset with no strong increments.
There is a subset $\Delta^{\prime}$ of $\Delta$ so that $\Delta^{\prime}$
is a robust additively non-smoothing set of strength $O( \epsilon)$. 
\end{proposition}

Returning to the setting of robust additively non-smoothing sets, we let, for the remainder of the section, the
set $\Delta$ be a robust additively non-smoothing set of strength $\delta$.

Given a value $x \in \Delta-\Delta$, we define $m(x)$ to be the number of pairs $(a,b) \in \Delta \times \Delta$
so that $a-b = x$. Clearly we have
$$E_4(\Delta)=\sum_{x \in \Delta-\Delta}  m(x)^2.$$

Given a robust additively non-smoothing set $\Delta$ of strength $\delta$, for each $\alpha$, we may define
$G_{\alpha} \subset \Delta \times \Delta,$
by
$$G_{\alpha} = \{ (a,b) \in \Delta \times \Delta :   N^{1+\alpha} \leq  m(a,b) < 2 N^{1+\alpha} \}.$$
By the dyadic pigeonhole principle, there is an $\alpha$ so that
$$|G_{\alpha}| \gtrsim { N^{6-\alpha - \delta} \over \log N}.$$

Moreover, we know that no $a$ in $\Delta$ participates in more that $N^{4+\delta}$ quadruples.
Thus no element $a$ in $\Delta$ participates in more than $N^{3-\alpha+\delta}$ pairs
in $G_{\alpha}$. Thus there are at least $\sim { N^{3-2\delta} \over \log N}$ elements of $\Delta$ each of which participate
in at least $\sim { N^{3-\alpha-2\delta} \over \log N}$ pairs in $G_{\alpha}$, by the large families principle (Lemma \ref{largefamilies}).

We now forget about optimizing our exponents and consolidate this information in a single definition.

\begin{definition}
We say that $(\Delta,G,D)$ is an additive structure at height $\alpha$ with ambiguity $\eta$ if the following hold.
We have
$$|\Delta| \leq N^{3 + \eta}.$$
We have
$$G \subset  \Delta \times \Delta$$
with the property that for each $(a,b) \in G$ we have that $a-b \in D$, and so that each $d \in D$ has $\sim N^{1+\alpha}$
representations as a difference of a pair in $G$. We have $|G| \sim N^{6-\alpha-\eta}$. Moreover
there are at least $N^{3-\eta}$ elements of $\Delta$ participating in at least $N^{3-\alpha-\eta}$ sums each.
Finally there are no more than $N^{15 + \eta}$ additive octuples among elements of $\Delta$.
\end{definition}
We summarize what we have shown so far in a proposition.

\begin{proposition} \label{structure1} Given a robust additively non-smoothing set $\Delta$ of strength $\delta$
we may find $G \subset \Delta \times \Delta$, and $D \subset \Delta-\Delta$ and $\alpha \geq 0$ so that
$(\Delta, G,D)$ is an additive structure at height $\alpha$ and ambiguity $O(\delta)$.
\end{proposition}

We now describe a slightly deeper property of additive structures at height $\alpha$ and ambiguity $\eta$.
Given a structure $(\Delta,G,D)$, for each $x \in D$, we define the set $\Delta_G[x]$ to be the set
of $a \in \Delta$ so that there exists $b \in \Delta$ with $(a,b) \in G$ and $a-b = x$. In light of
our definitions, we have for each $x \in D$ that $|\Delta_G[x]| \sim N^{1+\alpha}$.  We consider the quantity

\begin{equation} \label{Komity} K(\Delta,G,D)= \sum_{x \in D} \sum_{y \in D}  | \Delta_G[x] \cap \Delta_G[y] |. 
\end{equation}

Clearly $K(\Delta,G,D)$ counts the number of triples $(a,x,y)$ with $a \in \Delta_G[x]$ and $a \in \Delta_G[y]$.
Each element in $a$ is contained in exactly as many sets $\Delta_G[x]$ as it participates (in the first position)
in pairs in $G$. Thus we conclude that for $(\Delta,G,D)$ an additive structure with height $\alpha$
and ambiguity $\eta$ that
$$K(\Delta,G,D) = \sum _{a\in \Delta } |\{ x\in D \colon a \in \Delta _G [x] \} |^2 \gtrsim N^{9-2\alpha - 3\eta}.$$

Now examining the equation (\ref{Komity}) and dyadically pigeonholing, we observe that we can find $\beta$
so that there are  at least $N^{9-2\alpha - \beta - 4\eta}$ pairs $(x,y)$ so that for each such pair,
we have $N^{\beta} \leq |\Delta_G[x] \cap \Delta_G[y]| < 2N^{\beta}$.

\begin{definition} 
We say that the additive structure $(\Delta,G,D)$ at height $\alpha$ and with ambiguity $\eta$ has comity $\mu$
if we can find the abovementioned $\beta$ with $\beta > 1 + \alpha - \mu$.  (Note that of course $\beta \leq 1+ \alpha$ because 
$|\Delta_G[x] \cap \Delta_G[y]| \leq |\Delta_G[x] | \lesssim N^{1+ \alpha} $.)
\end{definition}

\begin{lemma} \label{Iteration} Given an additive structure $(\Delta,G,D)$ at height $\alpha$ and with ambiguity
$\eta$ either it has comity $\mu$ or there is an additive structure $(\Delta,G^{\prime},D^{\prime})$ with height
$\beta-1 < \alpha-\mu$ and ambiguity $O(\eta)$. \end{lemma}
We remark that this lemma contains the key use of the nonsmoothing hypothesis, which is hidden in the definition of ``additive structure".
\begin{proof}

We dyadically pigeonhole the equation (\ref{Komity}) to find $\beta$ so that there is a set
of at least $N^{9-2\alpha - \beta - 4\eta}$ pairs $(x,y)$ so that for each such pair,
we have $N^{\beta} \leq |\Delta_G[x] \cap \Delta_G[y]| < 2N^{\beta}$.
If it happens that $\beta-1 > \alpha-\mu$, then we are done. Otherwise, we will construct an
additive structure at height $\beta-1$.

Now any time that $a \in \Delta_G[x] \cap \Delta_G[y]$, this means that we can write $x=a-b$ and $y=a-c$.
Thus we have $x-y=c-b$. We have between $N^{\beta}$ and $2 N^{\beta}$ representations of the difference
$x-y$. It remains to determine how many such differences there are.

We have two distinct upper bounds on the number of such differences.  First there are  
$\lesssim N^{6-\beta + 2\eta}$,
since each difference is represented by $\sim N^{\beta}$ pairs in $\Delta \times \Delta$ and there
are only $N^{6+2\eta}$ such pairs. The second estimate is that there are  
$\lesssim N^{7-2\beta + O(\eta)}$
many such differences, because otherwise $E_4(\Delta)$ would be much larger than $N^7$ which would make
$E_8(\Delta)$ larger than $N^{15 + \eta}$. The first upper bound is most effective (ignoring ambiguity) when
$\beta < 1$ while the second is most effective when $\beta > 1$. Our plan (modulo ambiguity) is that
we shall rule out the case $\beta < 1$ and that we shall show that
the second upper bound is tight up to a factor of
$N^{O(\eta)}$. Both estimates will follow from the upper bound on $E_8(\Delta)$ and the Cauchy-Schwarz inequality,
namely Lemma \ref{cs}.

Since we have $N^{9-2\alpha-\beta -O(\eta)}$ pairs $(x,y)$ with at most $N^{6-\beta + O(\eta)}$ differences,
by the Cauchy-Schwarz inequality, there must be at least $N^{12-4\alpha-\beta-O(\eta)}$ additive quadruples in $D$, namely
$x-y=x^{\prime}-y^{\prime}$. (Here we let $S$ be the set of pairs $(x,y)$, we let $T$ be the set of differences
with $\sim N^{\beta}$ representations as difference of $\Delta$ and we let $\rho$ be the difference map,
$\rho(x,y)=x-y$. Then we can apply Lemma \ref{cs}.)
However since each difference $x,y$  can be represented in $N^{1+\alpha}$ ways as a difference in $\Delta$,
we can represent each quadruple in $D$ as an octuple in $\Delta$ in $N^{4+4\alpha}$ ways. Thus there are
at least $N^{16-\beta -O(\eta)}$ many such octuples which implies $\beta \geq 1 - O(\eta)$.

Thus we are in the regime where the estimate that there are at most $N^{7-2\beta + O(\eta)}$ many differences is 
most effective. Suppose that there were only $N^{7-2\beta - \gamma}$ many such differences with $\gamma >> \eta$
Then apply Cauchy-Schwarz again, we would see that there are at least $N^{11-4\alpha -O(\eta) + \gamma}$
many quadruples in $D$ which implies $N^{15-O(\eta) + \gamma}$ octuples in $\Delta$, a contradiction with the nonsmoothing hypothesis in the definition of ``additive structure".

Thus taking $D^{\prime}$ to be the differences $x-y$ obtained from $(x,y)$ so that 
$N^{\beta} \leq |\Delta[x] \cap \Delta[y]| < 2 N^{\beta}$ and taking $G^{\prime}$ to consist
of representatives of these differences coming from the intersections (as in the second paragraph of this proof), we obtain an additive structure
$(\Delta,G^{\prime},D^{\prime})$ with height
$ \beta - 1< \alpha-\mu$ and ambiguity $O(\eta)$.

\end{proof}

\begin{corollary} \label{comityatlast} Given an additive structure $(\Delta,G,D)$ at height $\alpha$ and with ambiguity
$\eta$ there is an additve structure $(\Delta,G,D)$ at height $\alpha^{\prime} \leq \alpha$ with ambiguity $\mu$ and
comity $\mu$  with $\mu \lesssim {1 \over  \log {1 \over \eta}}$. \end{corollary}

\begin{proof} We iteratively apply Lemma \ref{Iteration} with comity $\mu$ fixed by
$$\mu= {K \over \log {1 \over \eta}}$$
with $K$ a large constant,
and with the ambiguity increasing
by a constant factor $C$ in each iteration. Since $\alpha$ decreases by $\mu$ each time we don't find comity we need
only $\sim{1 \over \mu}$ iterations to achieve comity. At this point, we have ambiguity given by
$C^{{ \log ({1 \over \eta}) \over K} } \eta  << \mu$,
as long as $K$ was chosen sufficiently large.
\end{proof}

Now we begin to investigate what we can say about the shape of the set $H$ of 
all pairs $(b,c)$    in $\Delta \times \Delta$   having the property that $b-c$ has at least $N^{1+\alpha-O(\mu)}$
representations in $\Delta \times \Delta$
for $(\Delta,G,D)$ an additive structure
with height $\alpha$ and ambiguity and comity $\mu$. We will find that the set $H$ is rather thick in a product
set whose projection has size $N^{3-\alpha-O(\mu)}$. 

\begin{lemma} \label{blocks}  Let $(\Delta,G,D)$ be an additive structure
with height $\alpha$ and ambiguity and comity $\mu$. Then there is a subset $B \subset \Delta$ with
$|B| \gtrsim N^{3-\alpha-O(\mu)}$ so that there is a set $H \subset B \times B$ with 
$$|H| \gtrsim N^{-O(\mu)} |B|^2,$$
so that for any $(b,c) \in H$, the difference $b-c$ has $N^{1+\alpha - O(\mu)}$ representations in $\Delta \times
\Delta$.  \end{lemma}

\begin{proof}  From the hypotheses, we have that
$$\sum_{x \in D} \sum_{y \in D} |\Delta_G[x] \cap \Delta_G[y] | \gtrsim N^{9-2 \alpha - O(\mu)},$$
and that there are at least $N^{8-3\alpha -O(\mu)}$ pairs $(x,y)$ for which
$$|\Delta_G[x] \cap \Delta_G[y] | \gtrsim N^{1+\alpha - O(\mu)}.$$

Using the pigeonhole principle, we fix one value of $x$ for which there are $N^{3-\alpha - O(\mu)}$ choices of $y$
so that 
$$|\Delta_G[x] \cap \Delta_G[y] | \gtrsim N^{1+\alpha - O(\mu)}.$$

Again using the pigeonhole principle, we find an $a \in \Delta$ and a set $Y \subset D$ so that
$a \in \Delta_G[y]$ for every for every $y \in Y$, so that $|Y| = N^{3-\alpha - O(\mu)}$ and so that for each
$y \in Y$, we have
$$|\Delta_G[x] \cap \Delta_G[y] | \gtrsim N^{1+\alpha - O(\mu)}.$$

We notice that by definition $a-Y \subset \Delta$. We choose $B=a-Y$
Finally since each $\Delta_G[y]$ has relative density $N^{-O(\mu)}$
in $\Delta_G[x]$, we have by Cauchy-Schwarz (Lemma \ref{Carbery}) that
$$\sum_{y \in Y}  \sum_{y^{\prime} \in Y} |\Delta_G[y] \cap \Delta_G[y^{\prime}]| \gtrsim |Y|^2 N^{1+ \alpha - O(\mu)}.$$
This implies that $B$ satisfies the conclusion of the lemma. The reason is that by Lemma \ref{largefamilies},
we have a set $\tilde Y$ of pairs $y,y^{\prime}$ so that $|\Delta[y] \cap \Delta[y^{\prime}]| \gtrsim
N^{1+\alpha-O(\mu)}$. This implies that $a-y^{\prime} - (a-y) = y-y^{\prime}$ has at least $N^{1+\alpha-O(\mu)}$ representatives
as a difference of two elements of $\Delta$.
\end{proof}

Now we are going to use Lemma \ref{blocks} repeatedly to show that for any robust additively non-smoothing
set of size $\delta$ we can find an additive structure  of ambiguity $\eta$ with 
$\eta \lesssim {1 \over  \log {1 \over \delta}}$ which breaks into dense blocks.

\begin{lemma} \label{chunkiness} Let $\Delta$ be a robust additively non-smoothing set of strength $\delta$.
Choose $\mu \sim {1 \over  \log {1 \over \delta}}$. Then for some $0 \leq \alpha \leq 1$, there is
an additive structure $(\Delta, G, D)$ of height $\alpha$ and
ambiguity $\mu$ and disjoint subsets $B_1,\dots B_K$ of $\Delta$
with each $B_j$ satisfying $|B_j| \lesssim N^{3-{\alpha} + O(\mu)}$ so that
$$G \subset \bigcup_{j=1}^K B_j \times B_j.$$
Note that since we are requiring that $(\Delta, G, D)$ be an additive structure, this requires
$|G| \gtrsim N^{6-\alpha - O(\mu)}$ which implies $K \gtrsim N^{\alpha-O(\mu)}$.
\end{lemma}

\begin{proof} Using Proposition \ref{structure1} , Corollary \ref{comityatlast}, and Lemma \ref{blocks}.
We can find a subset $B_1$ of $\Delta$ so that for some choice of $\alpha$, it has size at most
$N^{3-\alpha+O(\mu)}$ and nevertheless $B_1 \times B_1$ contains at least $N^{6-2\alpha-O(\mu)}$ pairs
whose differences have $N^{1+\alpha-O(\mu)}$ representations as differences in $\Delta$.

Having done this, we use the robustness property of $\Delta$ to apply the same argument to $\Delta \backslash
(B_1 \cup -B_1)$. We continue removing sets from $\Delta$ until we have exhausted half of $\Delta$. Now one difficulty
is that the disjoint sets $B$ which we chose do not all have the same $\alpha$. We use dyadic pigeonholing
to resolve this for only a small cost in the number of sets. We  call these sets
$B_1,\dots, B_K$. Now Lemma \ref{blocks} guarantees us in each $B_j \times B_j$, a subset $H_j$
of cardinality at least $N^{6-2\alpha - O(\mu)}$ so that each difference in $H_j$ is represented
in $\Delta \times \Delta$ at least $N^{1+\alpha-O(\mu)}$ times. We denote by $D_{\alpha}$ the set of
differences represented in $\Delta \times \Delta$ at least $N^{1+\alpha-O(\mu)}$ times, and note that
$$|D_{\alpha}| \lesssim N^{5-2\alpha+ O(\mu)},$$
lest there be enough quadruples in $\Delta$ to violate the additive non-smoothing condition.
Using the large families principle and pigeonholing, we find some $\alpha^{\prime} \gtrsim \alpha -O(\mu)$ so that at least $N^{5-2\alpha-O(\mu)}$
differences are represented at least $N^{1+\alpha^{\prime}}$ many times in $\cup_j H_j $. We denote
this set of differences as $D^{\prime}_{\alpha}$. We let $D$ be $D^{\prime}_{\alpha}$ and let $G$
be a subset of $\cup_j  H_j $ consisting of $N^{1+\alpha^{\prime}}$ representatives of
each difference in $D$.

\end{proof}

Our goal now will be to use will be to use Lemma \ref{chunkiness} to find almost additively closed sets $E$ of size at least
$N^{1-f(\delta)}$ inside robust non-additively smoothing sets of strength $\delta$. Here $f:[0,1] \longrightarrow
[0,\infty)$ is some function with $\lim_{t \longrightarrow 0} f(t) =0$. We will be employing such functions
from now on in the paper. They, like constants, will change from line to line.

The project of finding additively closed sets will be easiest when we have
additive structures of height essentially zero having ambiguity and comity $\mu$. For this reason, we are about to define a stylized 
structure which generalizes this situation. We will eventually use the generalized version, replacing
$\Delta$ with the blocks $B_j$.

We will now define a $\mu$-full stylized $\rho$-structure which is $\tau$-energetic and has ambiguity and comity $\mu$. 
(The error exponents $\mu$ are all the same.) This will be a set $(\Delta^{\prime},G,D)$
where 
$$|\Delta^{\prime}| \sim N^{\rho},$$
(hence a $\rho$-structure)
where $G \subset \Delta^{\prime} \times \Delta^{\prime}$ with
$$|G| > N^{2 \rho - O(\mu)},$$
(this was the $\mu$-fullness), where $D$ is the set of differences in pairs in $G$ and each
difference represented $\gtrsim N^{\tau}$ and $\lesssim N^{\tau + O(\mu)}$ times, hence $\tau$-energetic. 
Finally we assume that there are  at least $N^{3 \rho - \tau - O(\mu)}$
pairs $(x,y) \in D \times D$ so that 
$$|\Delta^{\prime}[x] \cap \Delta^{\prime}[y] | \gtrsim  N^{\tau - O(\mu)},$$
which is of course the $\mu$-comity. 

We shall say that a set $K$ is $\mu$-additively closed provided that
$$|K-K| \lesssim N^{O(\mu)} |K|,$$ 
as in Section \ref{additivelyclosed}.

\begin{lemma} \label{kickass} There is a function $f:[0,1] \longrightarrow [0,\infty)$ with
$$\lim_{t \longrightarrow 0} f(t)=0,$$
so that the following holds.
Let $(\Delta^{\prime},G,D)$ be a $\mu$-full stylized $\rho$-structure which is 
$\tau$-energetic and has ambiguity and comity $\mu$. Then there is an $f(\mu)$ additively closed set $K$
with
$$  |K| \gtrsim N^{\tau-f(\mu)} $$
and a set $X$ so that
$$|X| \lesssim N^{f(\mu)} {|\Delta^{\prime}| \over |K|},$$
so that
$$|\Delta^{\prime}  \cap ( X+K)| \gtrsim N^{\rho-f(\mu)}.$$
\end{lemma}

\begin{proof} We proceed essentially as in the proof of Lemma \ref{blocks}. We find $x \in D$ so that there
is a set $Y$ of $y \in D$ with $|Y| \gtrsim N^{\rho-O(\mu)}$ so that 
$$|\Delta^{\prime}_G [x] \cap \Delta^{\prime}_G [y]| \gtrsim N^{\tau - O(\mu)},$$
for every $y \in Y$. As before, we use the pigeonhole principle to find $a \in \Delta^{\prime}$
so that there is a subset $Y_a$ of $Y$ so that for each $y \in Y_a$, we have
$$a \in \Delta^{\prime}_G[y]$$
and so that
$$|Y_a| \gtrsim N^{\rho-O(\mu)}.$$
However $Y_a \subset a-\Delta^{\prime}$.  Thus we think of $Y_a$ as a dense part of a translate of $-\Delta^{\prime}$.
Now we know that
$$\sum_{y \in Y_a} |\Delta^{\prime}_G[x] \cap \Delta^{\prime}_G[y]| \gtrsim N^{\rho + \tau - O(\mu)}.$$

This precisely means that there are $N^{\rho + \tau - O(\mu)}$ triples $(b,c,d)$ with $b,c \in \Delta^{\prime}$ and 
$d \in \Delta^{\prime}[x]$ with $a-b=d-c$ with $a$ still fixed. Applying Cauchy-Schwarz (see Lemma \ref{cs}) we find
$N^{\rho + 2 \tau -O(\mu)}$ quadruples $(d,c,d^{\prime},c^{\prime})$ with $d,d^{\prime} \in \Delta^{\prime}[x]$
and $c,c^{\prime} \in \Delta^{\prime}$. As it happens, this is precisely the hypothesis of the 
asymmetric Balog-Szemeredi-Gowers theorem (Theorem \ref{absg})  applied to $\Delta^{\prime}$
and $\Delta_G[x]$. The conclusion  follows directly. 

\end{proof}

We are now prepared to state the main result of this section.

\begin{theorem} \label{structuretheory}  Let $\Delta$ be a robust additively non-smoothing set of strength $\delta$.
As before choose $\mu \sim {1 \over  \log {1 \over \delta}}$. There is $f:[0,1] \longrightarrow [0,\infty)$ with
$$\lim_{t \longrightarrow 0} f(t) =0,$$
so that for some $\gamma \geq 0$, there
is an $f(\mu)$-additively closed set $K$ with 
$$|K| \gtrsim N^{1+\gamma - f(\mu)},$$
contained in $\Delta$.
In the event that we must have $\gamma=O(f(\mu))$,
for some $0 \leq \alpha \leq 1$ , we may find pairwise disjoint subsets $B_1,\dots,B_M \subset \Delta$ with 
$M \gtrsim N^{\alpha - O(\mu)}$
so that for each integer $1 \leq j \leq m$, we have
$$N^{3-\alpha - O(\mu)} \lesssim |B_j| \lesssim N^{3-\alpha + O(\mu)},$$
and moreover we find for each $j$ a $\mu$-additively closed set $K_j$ with
$$N^{1 - f(\mu)} \lesssim |K_j| \lesssim N^{1 + f(\mu)},$$
together with a set $X_j$ with
$$N^{2-\alpha - f(\mu)} \lesssim |X_j| \lesssim N^{2-\alpha+ f(\mu)},$$
so that
$$|B_j \cap (X_j+K_j)| \gtrsim N^{3-\alpha-f(\mu)}.$$
Further, there is a set $D$ with $|D| \gtrsim N^{5-2\alpha - f(\mu)}$ so that
each element of $D$ has at least $N^{1+\alpha-f(\mu)}$ representations as a difference of elements of $\Delta$
and so that for each $j$, the set of 4-tuples $Q_j=\{(k_1,d_1,k_2,d_2): k_1,k_2 \in K_j,d_1,d_2 \in D: k_1-d_1
=k_2-d_2\}$ satisfies
$$|Q_j|\gtrsim N^{7-2\alpha-f(\mu)}.$$
Moreover we may choose $K_j$ to be contained in the set of differences  having at least
$N^{5-2\alpha-f(\mu)}$ representations as a difference between elements of $D$.
\end{theorem}

\begin{proof} We apply Lemma \ref{chunkiness} and restrict our attention to
$$G \subset \bigcup_{j=1}^K B_j \times B_j$$
where the $B_j$ are the blocks obtained there. 
Now to $G$, we apply the argument used in
the proof of Lemma \ref{Iteration}.

That is, for any element $x \in -(G)$, we study
$$\Delta_G[x] = \{a \in \Delta:  \exists b \in \Delta:  (a,b) \in G;  \quad a-b=x \}.$$
Here $-(G) = \{ a-b \colon (a,b)\in G\}$.
Then, we observe
$$\sum_{x,y \in -(G)} |\Delta_G[x] \cap \Delta_G[y]| \gtrsim N^{9-2\alpha - O(\mu)},$$
and we observe as before that there is some $1+\alpha \geq \beta \geq 1-O(\mu)$ so that there are at least
$N^{9 - 2 \alpha - \beta-O(\mu)}$ pairs $(x,y)$ for which $|\Delta_G[x] \cap \Delta_G[y]| \gtrsim N^{\beta}$.
We note that when this happens, for each $a$ in the intersection $\Delta_G[x] \cap \Delta_G[y]$, we have
$x=a-b_1$ and $y=a-b_2$. Here if $a \in B_j$, we must have $b_1 \in B_j$ and $b_2 \in B_j$, since
$(a,b_1),(a,b_2) \in G$. We argue as in the proof of Lemma \ref{Iteration}, that there are
at least $N^{7-2\beta -O(\mu)}$ differences having $N^{\beta}$ representations in 
$$\bigcup_{j=1}^K  B_j \times B_j .$$
If not, $E_8 (\Delta ) >> N^{15 + \delta } $.  But 
\beqa
N^{\beta} N^{7-2\beta - O(\mu)} \leq \left|\bigcup_{j=1}^K  B_j \times B_j \right| \lesssim N^{6-\alpha + O(\mu)}, 
\eeqa
which implies $7-\beta \leq 6 - \alpha + O(\mu)$, which implies $1+\alpha - O(\mu)\leq \beta$.  
This gives us $O(\mu)$ comity for $G$.

Thus we have a set $H$ of pairs $x,y$ for which $|\Delta_G[x] \cap \Delta_G[y]| \gtrsim N^{1+ \alpha-O(\mu)}$
and so that $|H| \gtrsim N^{8-3\alpha-O(\mu)}$. Now we use the pairwise disjointness of the blocks
$B_j$ to write the identity
\begin{equation} \label{spreadthewealth}  \sum_{(x,y) \in H} |\Delta_G[x] \cap \Delta_G[y]| 
=\sum_{j=1}^K \sum_{(x,y) \in H} |\Delta_G[x] \cap \Delta_G[y] \cap B_j | \gtrsim N^{9-2\alpha - O(\mu)}
\end{equation}

Now we begin to use the comity of $G$.
We first eliminate from the second sum in equation \ref{spreadthewealth} all terms for which
the relative density of $\Delta_G[x]$ in $\Delta_G[y] \cap B_j$ or the relative density of 
$\Delta_G[y]$ in $\Delta_G[x] \cap B_j$ is smaller than $N^{-C \mu}$ for too large a constant $C$, again using the large families principle. By choosing
$C$ sufficiently large we do not reduce the sum by a factor of more than $2$. We dyadically pigeonhole
to obtain the largest possible sum from those terms where  
$N^{1+\gamma} \leq |\Delta_G[x] \cap \Delta_G[y] \cap B_j | \leq 2N^{1+\gamma}$. (We denote this set of $(x,y)$ as
$H_{\gamma,j}$), Thus we have  reduced the sum
by at most a factor of $\log N$. We keep only those $j$ for which 
\begin{equation} \label{yourok} \sum_{(x,y) \in H_{\gamma,j}} |\Delta_G[x] \cap \Delta_G[y] \cap B_j |
\gtrsim N^{9-3\alpha - O(\mu)}, \end{equation}
which we can do without sacrificing much by using again the large families principle.

We observe that for each $x$, there are at most $N^{3-\alpha+O(\mu)}$ choices of $y$ so that $(x,y) \in H$.
The reason is that  any $a \in \Delta$ belongs to at most $N^{3-\alpha + O(\mu)}$ sets $\Delta_G[y]$ 
because $G$ is contained in $\cup_{j} B_j \times B_j$. However if there were more than
$N^{3-\alpha+O(\mu)}$ choices of $y$ so that $(x,y) \in H$ then there would be elements $a \in \Delta_G[x]$
which are contained in $\Delta_G[y]$ for more than  $N^{3-\alpha+O(\mu)}$ choices of $y$
Thus, since we have at least $N^{8-2\alpha-\gamma-O(\mu)}$ triples $(j,x,y)$ for which
$$|\Delta_G[x] \cap \Delta_G[y] \cap B_j| \gtrsim N^{1+ \gamma},$$
while at the same time
$$|\Delta_G[x] \cap B_j| \lesssim N^{1+\gamma + O(\mu)},$$
(by the relative density of $\Delta _G [y]$ in $\Delta _G [x] \cap B_j $,)
it must be that for most values of $j$, we have at least $N^{8-3\alpha - \gamma}$ pairs $(x,y) \in H_{\gamma,j}$.
This means, fixing one such value of $j$ (since at most $N^{3-\alpha + O(\mu)}$ values of $y$ are paired with a given $x$) ,
there are at least $N^{5-2\alpha - \gamma -O(\mu)}$ differences $x$ with $N^{1+\gamma}$ representations
in $G \cap (B_j \times B_j)$. We call this set  $D_{j,\gamma}$.

Thus $G \cap (B_j \times B_j)$ is $\mu$-full and $1+\gamma$-energetic.
Another way of describing the $\mu$-fullness is that $N^{5-2\alpha - \gamma -O(\mu)}$ (up to $N^{O(\mu)}$ factors) is the
largest number of such $x$ possible, purely based on the size of $B_j \times B_j$. Thus it must be
that for a set of size $N^{5-2\alpha - \gamma -O(\mu)}$ many such $x$, there are $N^{3-\alpha-O(\mu}$ such $y$ with
$(x,y) \in H_{\gamma,j}$ (yet again, by the large families principle). Thus $(B_j,G \cap (B_j \times B_j),D_{j,\gamma})$ has $O(\mu)$-comity.  Clearly 
$(B_j, G \cap (B_j \times B_j), D_{j,\gamma})$ is a $3-\alpha$ structure with ambiguity $\mu$.
Thus we are in a position to apply Lemma \ref{kickass}. This proves the first part of the theorem.
Indeed, since all our estimates were optimal up to $N^{O(\mu)}$ factors, there is a set $J$ of choices
of $j$ for which we could apply Lemma \ref{kickass} with $|J| \gtrsim N^{\alpha - O(\mu)}$.

To prove the second part, we consider in detail the case $\gamma=0$. We will apply the
argument proving Lemma \ref{kickass}
to all $j \in J$.  This will give us $\mu$-additive sets $K_j$ and sets $X_j$ with appropriate
upper and lower bounds since we can assume $\Delta$ contains no $\mu$-additive sets with more than
$N^{1+f(\mu)}$ elements.

We will allow $f$ to vary from line to line and we will express even quantities that are
clearly $O(\mu)$ as $f(\mu)$.

We let $D$ be the set of all
differences $x$ for which $|\Delta[x] \cap B_j| \gtrsim N^{1-f(\mu)}$  for at least $N^{\alpha-f(\mu)}$
values of $j \in J$. 
For each value of $j \in J$ , there are at least
$N^{8-3\alpha -f(\mu)}$ pairs $(x,y) \in D^2$ with $|\Delta_G[x] \cap \Delta_G[y] \cap B_j | \gtrsim N^{1 - f(\mu)}$.
Note that we may also restrict $D$ to differences which cannot be represented in more than $N^{1+f(\mu)}$ ways
as differences of elements of $B_j$ for more than $N^{\alpha-f(\mu)}$ values in $j$ and so that in each $B_j$ our
count of good pairs $(x,y)$ consists only of pairs of differences which cannot be represent as differences in $B_j$
in more than $N^{1+f(\mu)}$ ways. Otherwise, we could choose $\gamma > f(\mu)$.

Now we recall the structure of the argument in Lemma \ref{kickass}.  We chose an $a \in B_j$ and a set $B_a$ of size
$N^{3-\alpha-f(\mu)}$ of the differences in which $a$ participates, and a set $K_a$ which is actually of the form
$\Delta_G[x] \cap B_j$ and has size $N^{1-f(\mu)}$. We find $N^{5-\alpha -f(\mu)}$ additive quadruples made up
of two elements of $B_a$ and two elements of $K_a$. We may strip down $K_a$ further to those
elements which participate in at least $N^{4-\alpha-f(\mu)}$ of these quadruples and not harm our estimate on the number of quadruples between $K_a$ and $B_a$.
Now, we note that since $B_a$ is a large subset of a translate of $B_j$, it must be that there are $N^{2-f(\mu)}$
pairs $(q_1,q_2) \in K_a^2$ with the property that $q_1-q_2$ is represented $N^{3-\alpha-f(\mu)}$ times
as a difference of elements of $B_j$. We let $K_{1,j}$ be the set of differences of $B_j$ that can be represented
in $N^{3-\alpha-f(\mu)}$ ways as differences in $B_j$.  Because $B_j$ contains
no $\mu$-additively closed set of size more than $N^{1+f(\mu)}$, we have that $|K_{1,j}| \leq N^{1+f(\mu)}$.
Otherwise we could apply the asymmetric Balog-Szemeredi-Gowers theorem, to obtain a $\mu$-additively
closed set, contained in $B_j$, as in the $\gamma >> O(f(\mu)) $ case.

We can replace $N^{5-\alpha-f(\mu)}$ quadruples $q_1-q_2=x_1-x_2$ with $q_1,q_2 \in K_{1,j}$ and $x_1,x_2 \in B_a$ by an 
equation of the
form $q=x_1-x_2$ with multiplicity  at most $N^{1+f(\mu)}$. Thus we obtain at least $N^{4-\alpha-f(\mu)}$ such equations.
We rewrite the equation as $x_1-q=x_2$ and use Cauchy Schwarz to obtain $N^{5-\alpha - f(\mu)}$ quadruples
$x_1-q=x_1^{\prime} - q^{\prime}$. We see then that without losing more than $N^{f(\mu)}$ factors in our estimates,
we can replace $K_a$ by $K_{1,j}$. This is good since we have made it independent of the choice of $a$. Now we need only
show that we can find many quadruples not only between $K_{1,j}$ and $B_a$ but between $K_{1,j}$ and $D$. This will give us 
the desired result.

To do this, we observe that we may delete from $B_a$ a set with relative density $N^{-f(\mu)}$ without harming our
estimates on the number of quadruplets between $B_a$ and $K_{1,j}$. Our goal will be to cover a subset $D^{\prime}$
by a disjoint union of subsets of the form $B^{\prime}_a$ where $B^{\prime}_a$ is a subset of $B_a$ with
relative density $1-N^{-f(\mu)}$. To do this we observe that for any fixed $a_1 \in B_j$
$$\sum_{a_2} |B_{a_1} \cap B_{a_2} | \lesssim N^{4-\alpha+f(\mu)}.$$
We can do this because the sum counts triples $(x,a_1,a_2)$ with $x$ a difference and $a,a_2$ are parts of representations
of it. We have assumed that we are only dealing with differences with fewer than $N^{1+O(\mu)}$ representations.
Here we are using that we are in the $\gamma=0$ case,

Now we produce $D^{\prime}$ as follows. We choose $a_1$ and keep the set $B_{a_1}$.
We add all elements of $B_{a_1}$ to $D^{\prime}$ We choose $B_{a_2}$ to have
the minimal possible sized intersection with $B_{a_1}$ and let $B_{a_2}^{\prime}$ be those elements
in $B_{a_2}$ that are not already in $D^{\prime}$. We choose $B_{a_3}$ to have minimal possible
intersection with $B_{a_1} \cup B_{a_2}$. We continue in this way until we reach a $k$ so that
$B_{a_k}^{\prime}$ no longer has relative density $1-N^{-f(\mu)}$ in $B_{a_k}$. Because the
average intersection $|B_a \cap B_{a^{\prime}}|$ is bounded by $N^{1 + f(\mu)}$, we get
that $k$ is at least $N^{2-\alpha -f(\mu)}$. Thus our set $D^{\prime}$ has relative density at least 
$N^{-f(\mu)}$ in $D$. Thus we have that $K_{1,j}$ has $N^{7-2\alpha-f(\mu)}$ quadruples with $D^{\prime}$ and {\it a fortiori}
with $D$. 

Now we slightly refine $K_{1,j}$ to $K_{2,j}$ consisting only to differences of elements
of $K_{1,j}$ which participate in
at least $N^{5-\alpha-f(\mu)}$ of the quadruples with $D^{\prime}$. 
(We perform this refinement in order to prove the very last claim of the theorem.)
Since $D^{\prime}$ is a disjoint
union of sets with relative density $N^{-f(\mu)}$ inside translates of $B_j$, it must
be that $K_{2,j}$ still has $N^{5-\alpha - f(\mu)}$ quadruples with $B_j$. Thus we can apply the
asymmetric Balog Szemeredi Gowers theorem to find a subset $K_j$ satisfying the conclusions of the
theorem.

\end{proof}

\section{Structure of spectrum of large capsets with no strong increments}

In this section, we transfer the results obtained in Theorem \ref{structuretheory} over to the
setting of the spectrum of large capsets with no strong increments. This turns out to be rather
simple. The main ideas which we have not yet taken advantage of are Freiman's theorem and the use of
the estimate in Proposition \ref{first} which bounds the number of elements of the spectrum in
a subspace of dimension $d$ by $d N^{1+2\epsilon}$.

We state the main result of the section.

\begin{theorem} \label{specstruct} Let $\Delta$ be the spectrum of large capset without strong increments.
There is a function $f:[0,1] \longrightarrow [0,\infty]$ with $\lim_{t \longrightarrow 0} f(t)=0$ so that 
the following holds.
There is a subspace $H$ of $F_3^N$ of dimension $N^{ f(\epsilon) }$ and a set $\Lambda \subset F_3^N$ of size 
$N^{2-f(\epsilon)}$ so that
for each element $\lambda \in \Lambda$, there is a subset $H_{\lambda} \subset H$ with the properties
that 
$$|H_{\lambda}| \gtrsim N^{1-f(\epsilon)},$$
and the sets $\lambda+H_{\lambda}$ are pairwise disjoint subsets of $\Delta$.
For any subspace $W \subset F_3^N$ of dimension $d$, we have that $W$ contains at most $d N^{f(\epsilon)}$ elements
of $\Lambda$. \end{theorem}

\begin{proof} As before we allow our function $f$ to vary from line to line until we achieve the desired result.

In light of Theorem \ref{Freiman} and Proposition \ref{first}, any $f(\epsilon)$-additively closed
set in which the spectrum has $N^{-f(\epsilon)}$ relative density, must be bounded in size by $N^{1+f(\epsilon)}$.
Therefore, we are in the $\gamma=0$ case of Theorem \ref{structuretheory}. We know that there is a set $\Delta^{\prime}$
of density $N^{-f(\epsilon)}$ in the spectrum $\Delta$ which is contained in
$$\bigcup_{j=1}^{M }  X_j + K_j,$$
with $M \gtrsim N^{\alpha - O(f(\epsilon))}$,
with each $K_j$ an $f(\epsilon)$-additively closed set (of size at least  $N^{1-f(\epsilon)}$ and at most
$N^{1+f(\epsilon)}) $ and with each set $X_j$ of size
  $ N^{2-\alpha \pm f(\epsilon)}$ . Moreover, each set $K_j$ lies in the 
set $K$ of differences having
at least $N^{5-2\alpha-f(\epsilon)}$ representations as differences of elements of $D$, the differences among elements
of the spectrum which have at least $N^{1+\alpha -f(\epsilon)}$ representations. In light of the non-additive
smoothing property of $\Delta$, we have that $|K| \lesssim N^{1+f(\epsilon)}$ since there can be at most 
$N^{11-4\alpha +
f(\epsilon)}$ quadruplets among elements of $D$. We may eliminate all elements $q$ of each $K_j$ for which
$\Delta \cap q+ X_j$ does not have size at least $N^{2-\alpha -f(\epsilon)}$. Now we let $K^{\prime}$ be the set of
elements of $K$ which appear in at least $N^{\alpha-f(\epsilon)}$ many $K_j$. We can find some $K_j$ which has intersection
of size $N^{1-f(\epsilon)}$ with $K^{\prime}$. Then $K^{\prime} \cap K_j = K^{\prime \prime}$ is a 
$f(\epsilon)$-additively closed
set with cardinality at least $N^{1-f(\epsilon)}$. Moreover each element of $K^{\prime \prime}$ is contained
in $N^{\alpha-f(\epsilon)}$ many $K_j$. Thus by pigeonholing there are at least $N^{\alpha-f(\epsilon)}$ many $K_j$ so that
$|K_j \cap K^{\prime \prime}| \gtrsim N^{1-f(\epsilon)}$. We only keep these $j$ and replace $K_j$ by 
$K_j \cap K^{\prime \prime}$. But by Theorem \ref{Freiman}, we have that $K^{\prime \prime}$ is contained
in a subspace of dimension $N^{f(\epsilon)}$ which we call $H$.

This basically proves the first part of the theorem. We have that a subset of the spectrum of density $N^{-f(\epsilon)}$
is contained in $K^{\prime \prime} + X$, where $X$ is the union of the $X_j$'s. We will pick $\Lambda$ and the sets
$H_{\lambda}$ as follows: Find $x_1$ in $X$ so that at least $N^{1-f(\epsilon)}$ elements of $\Delta$
are contained in $x_1+K^{\prime \prime}$. Let $\Delta_1$ be the elements of $\Delta$ contained in
$X+K^{\prime \prime}$ but not in $x_1+K^{\prime \prime}$. Let $\Delta^1$ be those elements  of $\Delta$ contained
in $x_1+K^{\prime \prime}$ and let $H_{x_1} = \Delta^{1} - x_1$. Note that $H_{x_1}$ is contained in $K^{\prime \prime}$
and therefore in $H$. Now we proceed iteratively. Find $x_j \in X$ so that there are at least $N^{1-f(\epsilon)}$
elements of $\Delta_{j-1}$ in $x_j+K^{\prime \prime}$. When this is no longer possible, we terminate the process.
Then we let $\Delta_j$ be the elements of $\Delta_{j-1}$ not in $x_j + K^{\prime \prime}$ and we let $\Delta^j$ be the
ones that are. We let $H_{x_j}= \Delta^j-x_j$. We let $\Lambda=\{ x_j\}$ after the iteration has terminated.  
Note that if 
\beqa
|(x_j + K'' ) \cap \Delta _{j-1} | << N^{1-f(\epsilon)} 
\eeqa
for all remaining $x_j$ in $X$, then 
\beqa
|\bigcup _{i=1} ^{j} (x_i + K'') | \gtrsim N^{3 - f(\epsilon)},
\eeqa
so $|\Lambda | \gtrsim N^{2-f(\epsilon)}$.

To prove the second part of the theorem, let $S$ be any subset of $\Lambda$ with some cardinality $M$. 
But $S+H$ contains at least $M N^{1-f(\epsilon)}$
elements of $\Delta$. This contradicts Proposition \ref{first} unless the span of $S$ has dimension
at least $M N^{f(\epsilon)}$.

\end{proof}

\section{Contradiction} \label{lastsection}

The goal of this section is to obtain a contradiction from the existence of large capsets without strong increments
by using the result of Theorem \ref{specstruct}. We begin by recording some easy consequences of Plancherel's identity
for the interaction between the Fourier transform of the characteristic function of a set and the Fourier transforms
of its fibers over a subspace.

For any set $A \subset F_3^N$, we define its Fourier transform
$$\hat A(x)={1 \over 3^N} \sum_{a \in A}  e(a \cdot x).$$
We state Plancherel's identity:

\begin{proposition} \label{Plancherel}
$$\sum_{x \in F_3^N}  |\hat A(x)|^2 = 3^{-N} |A|.$$
\end{proposition}

We let $H$ be a subspace of $F_3^N$ and we let $H^{\perp}$ be its annihilator. We let $V$ be a subspace of
the same dimension as $H$ which is transverse to $H^{\perp}$ i.e., $ V+ H^{\perp} = F_3^N $ .  We define the fiber $A_{H,v}$ for $v \in V$ by
$$A_{H,v} = A \cap (H^{\perp}+ v).$$

If we have $h \in H$, then
$$\hat A(h)= 3^{-N} \sum_{v \in V}  e(h \cdot v)  |A_{H,v}|.$$

Thus we arrive at another form of Plancherel:

\begin{proposition} \label{Plancherel2}
$$\sum_{h \in H} |\hat A(h)|^2  = |H| 3^{-2N} \sum_{v \in V} |A_{H,v}|^2.$$
Moreover
$$\sum_{h \neq 0 \in H}  |\hat A(h)|^2 = |H| 3^{-2N} \sum_{v \in V}  (|A_{H,v}| - {|A| \over |H|})^2 $$
\end{proposition}

Next, we consider the situation where we have a subspace $H \subset F_3^N$  and a larger subspace
$K$ with $H \subset K \subset F_3^N$. We let $V$ be a subspace transverse to $H^{\perp}$ as before
and we would like to consider the Fourier transforms of the fibers of $A$, namely the sets $A_{H,v}$.
We can think of each fiber $A_{H,v}$ as being identified with a subset of $H^{\perp}$ (by translation by $v$)
and of course $H^{\perp}$ can be identified with $F_3^N / H$. That is,
we define functions $A_{H,v} : F_3^N / H \longrightarrow {\bf C}$ by

$$\hat A_{H,v}(w) =  {|H| \over 3^N} \sum_{a \in A_{H,v}} e( (a-v) \cdot w).$$

The function $\hat A_{H,v}(w)$ is well defined on $F_3^N \backslash H$ since $a-v$ is in $H^{\perp}$. Next, we
write down a version of Proposition \ref{Plancherel2} which shows how the 
$L^2$ norms of the Fourier transforms of the fibers on $K \backslash H$ with the $L^2$ norm of the 
Fourier transform on $K$. We let $W$ be a subspace transverse to $K^{\perp}$ with $V \subset W$.

\begin{proposition} \label{martingale}
With $H$, $K$, $V$, and $W$ as above,
$$\sum_{k \neq 0 \in K}  |\hat A(k)|^2 = \sum_{h \neq 0 \in H} |\hat A(h)|^2 + {1 \over |H|} 
\sum_{v \in V} \sum_{k \neq 0 \in K / H} |\hat A_{H,v}(k)|^2.$$
\end{proposition}

\begin{proof} 
Since clearly we have $K^{\perp} \subset H^{\perp}$, there is a unique subspace $W^{\prime} \subset W$
with $W^{\prime}+K^{\perp}=H^{\perp}.$ We have $V+W^{\prime} = W$.

We consider the following function on $W$,
$$g(w)=|A_{K,w}|-{|A| \over |K|}.$$
Clearly, in light of the second part of Proposition \ref{Plancherel2}, the left hand side of the identity we are
trying to prove is the normalized square of the $L^2$ norm of the function $g(w)$.

We now break up $g$ as the sum of a function $g_0$ which is constant on translates of $W^{\prime}$
and functions $g_v$ with $v$ running over $V$ having mean zero and supported on the translate $v+W^{\prime}$ of
$W^{\prime}$.  Clearly the functions $g_0$ and $\{g_v\}$ are pairwise orthogonal. The first term on the right
hand side of the identity is the normalized square $L^2$ norm of the function $g_0$. The second term on the
right hand side represents the sum over $v$ of the normalized square $L^2$ norm of the functions $g_v$. The identity
is then an application of the Pythagorean theorem.
\end{proof}

(We remark that Proposition \ref{martingale} can simply be thought of as Plancherel for a ``local Fourier
transform" of $A$. Here, we localize to the translates of $H^{\perp}$.

Now we are prepared to apply Proposition \ref{martingale} to the setting in which $A$ is large capset without
strong increments and $H$ is the subspace given to us by Theorem \ref{specstruct}.

We let $A$ be a large capset with no strong increments. As usual, $f$ will be a function taking $[0,1]$ to
$[0,\infty]$ with $\lim_{t \longrightarrow 0} f(t)=0$. We will vary $f$ from line to line.

Then there is a subspace $H$ with dimension $N^{f(\epsilon)}$
and a set $\Lambda$ of size $N^{2-f(\epsilon)}$ so that for each $\lambda \in \Lambda$ there is a subset
$H_{\lambda}$ of $H$, so that $|H_{\lambda}| > N^{1-f(\epsilon)}$ so that for each $h \in H_{\lambda}$, we have that
$$|\hat A(h+\lambda)| \gtrsim N^{-1-f(\epsilon)} 3^{-N} |A|.$$
We also have that the sets $\lambda+H_{\lambda}$ are pairwise disjoint.

Note therefore that
$$\sum_{\lambda \in \Lambda} \sum_{h \in H_{\lambda}} |\hat A(h+\lambda)|^2 \gtrsim N^{-f(\epsilon)} N 3^{-2N} |A|^2 
\gtrsim N^{-f(\epsilon)} 3^{-N} |A|.$$

Thus the structured elements of the spectrum of $A$ account for a large proportion of the squared $L^2$ norm
of the Fourier transform of $A$.

Now we would like to consider the fibers $A_{H,v}$, where $H$ is the subspace we've been discussing.
Because the capset $A$ has no strong increments, we know that for each value $v$, we have
$$|A_{H,v}| \leq ({|A| \over |H|})  (1+ N^{f(\epsilon)-1}).$$
We will now momentarily fix the function $f$.

However, we don't have a good lower bound on $|A_{H,v}|$ in general. All we know is that sum of all the positive
increments is equal to the sum of all the negative increments.  (See the proof of Proposition \ref{first}.)  We let $V_{bad}$ be the set of all $v \in V$
for which
$$|A_{H,v}| \leq ({|A| \over |H|}) (1- N^{2 \sqrt{f(\epsilon)}-1}).$$
(That is $V_{bad}$ is the set of those $v \in V$ for the which the fiber has a bad negative increment.)
We know that
$$|V_{bad}| \lesssim N^{-\sqrt{f(\epsilon)}} |V|.$$
We define
$$A_{bad} = \bigcup_{v \in V_{bad}} A_{H,v},$$
and we let
$$A^{\prime} = A \backslash A_{bad}.$$
We know that 
$$|A_{bad}| \lesssim N^{-\sqrt{f(\epsilon)}} |A|.$$
Thus 
$$\sum_{x} |\hat A_{bad}(x)|^2 \lesssim N^{-{\sqrt{f(\epsilon)} \over 2}} 
\sum_{\lambda \in \Lambda} \sum_{h \in H_{\lambda}} |\hat A(h+\lambda)|^2.$$
Thus removing $A_{bad}$ does not perturb the large spectrum of $A$ too much.

(In making this precise, we now resume changing the function $f$ from line to line, observing that we may take the 
next $f$ to be larger than the previous $\sqrt{f}$.) We may find a set $\Lambda^{\prime}$ which is a subset of 
$\Lambda$ with
$| \Lambda^{\prime}| \gtrsim N^{2-f(\epsilon)}$ and so that for each $\lambda \in \Lambda^{\prime}$ there
is a subset $H^{\prime}_{\lambda}$ of $H$ so that for each $\lambda \in \Lambda^{\prime}$ and each 
$h \in H^{\prime}_{\lambda}$ we have
$$|\hat A^{\prime} (h+\lambda)| \gtrsim N^{-1-f(\epsilon)} 3^{-N} |A|.$$

Thus from the point of view of the structure of the spectrum, we have that $A^{\prime}$ is essentially
as good as $A$. However, the set $A^{\prime}$ has a big advantage over $A$ in that we have good bounds
on the Fourier transform of its fibers. This is because the fibers are either empty or close to size
${|A| \over H}$. Empty fibers achieve no increments. On the other hand, fibers which are close
to average cannot have an increment too large, or else the set $A$ will have a strong increment on a translate
of a codimension 1 subspace of $H^{\perp}$. Precisely, the estimate we get is
\beqan \label{newnumber}
|\hat A^{\prime}_{H,v} (x)| \lesssim  3^{-N} |A| N^{-1+ f(\epsilon)}= \rho N^{-1+ f(\epsilon)}.
\eeqan
(This is because the negative increment on density to pass from $A$ to $A^{\prime}_{H,v}$ does no
more than to reduce the density by a factor of $1-N^{f(\epsilon)-1}$ and the codimension of $H$ is only $f(\epsilon)$.)
Should the Fourier transform of $A^{\prime}_{H,v}$ exceed the above bound, then $A$ will have
a strong increment into a codimension one subspace of $H^{\perp} + v$.

Moreover, the set $A^{\prime}_{H,V}$ is a large capset without strong increments
(on subspaces of codimension no more than ${N \over 2} - N^{f(\epsilon)}$)
but with $\epsilon$ replaced
by $f(\epsilon)$. Thus we see using Proposition \ref{first} that for any subspace $L$ of $F_3^N \backslash H$
with dimension $d$, we have that
$$\sum_{x \in L}  |\hat A^{\prime}_{H,v}(x)|^2  \lesssim ({|A| \over 3^{N}})^2 d N^{-1+f(\epsilon)}.$$
This estimate is a version of the bound on the number of elements of the spectrum of $A^{\prime}_{H,V}$ in $W$.

From this information together with the fact that no subspace $K$ of dimension $d$ contains more than
$dN^{f(\epsilon)}$ elements of $\Lambda$ and hence of $\Lambda^{\prime}$, we are ready to achieve a contradiction.

We introduce some measure spaces on which we will apply Lemma \ref{Carbery}. For each $v \in V$, we define
$\chi_v$, a measure on $F_3^N \backslash H - \{0\}$ by
$$\chi_v(X)=\sum_{x \in X} |\hat A^{\prime}_{H,v} (x)|^2.$$

Clearly the total measure of $\chi_v$ is bounded by $3^{-N} |A| N^{f(\epsilon)}$, by Plancherel.
We now give a construction for sets which have large density for many of the measures $\chi_v$.

Let $\Xi \subset \Lambda^{\prime}$ with  $|\Xi| = d \lesssim N^{1-f(\epsilon)}$.
Then in light of the fact that $\Xi + H$ contains at least $d N^{1-f(\epsilon)}$ points $x$
where $A^{\prime}(x)$ is large and in light of Proposition \ref{martingale}, there
must be $\gtrsim N^{-f(\epsilon)} |H|$ values of $v$ for which $span(\Xi)$ has density at least $d N^{-2-f(\epsilon)}$
for $\chi_v$, by the large families principle. (Here, by slight abuse of notation, the span is taken in $F_3^N \backslash H$.)

We will take $d \sim N^{{2 \over 3}}$. Now let $\Pi$ be any collection of $N^{ {8 \over 3} + O(\sqrt{\mu})}$
subsets of $\Lambda^{\prime}$ with cardinality $N^{{2 \over 3}}$. We chose the exponents
${2 \over 3}$ and ${8 \over 3}$ somewhat
arbitrarily. In particular, they allow us to apply Lemma \ref{Carbery}.
 By pigeonholing, we may find a set of $v$'s of 
cardinality $|H| N^{-f(\epsilon)}$ for which at least $N^{ {8 \over 3} - f(\epsilon)}$ of the sets $\Xi$ in
$\Pi$ have the property that $span(\Xi)$ has density $N^{-{4 \over 3} - f(\epsilon)}$ for $\chi_v$.
From this and Lemma \ref{Carbery} we obtain the lower bound

\begin{equation} \label{lowerbound} \sum_{v \in V - V_{bad}} \sum_{\Xi_1 \in \Pi} \sum_{\Xi_2 \in \Pi}  
{\chi_v (span(\Xi_1) \cap span(\Xi_2)) \over
\chi_v (F_3^N \backslash H) }  \gtrsim |H| N^{{8 \over 3} - f(\epsilon) }. \end{equation}

What this says is that modulo $N^{f(\epsilon)}$ terms, the average density
(in $\chi_v$ measure) of an intersection
$span(\Xi_1) \cap span(\Xi_2)$ is at least approximately $N^{-{8 \over 3}}$. This is rather large
since the density of any single element of $F_3^N \backslash H$ is only approximately $N^{-3}$ in light
of the bound on the Fourier transforms of the fibers, which follows from \eqref{newnumber}.

We will now produce a random construction of $\Pi$. We choose each set of $\Pi$ uniformly at random. We will
write down an upper bound on the expected density in any $\chi_v$ of the intersection
$span(\Xi_1) \cap span(\Xi_2)$ which contradicts the lower bound \ref{lowerbound}.
We observe that if $\rho$ is the density in $\chi_v$ of $span(\Xi_1) \cap span(\Xi_2)$ then
$$\rho \leq N^{-3+O(f(\epsilon))} (3^k-1),$$
(because $\chi _v (x) \lesssim N^{-4 + O(f(\epsilon)}$,)
where $k$ is the nullity of the set $\Xi_1 \cup \Xi_2$. We recall that there are no more than 
$N^{{2 \over 3} + f(\epsilon)}$
elements of $\Lambda^{\prime}$ in any subspace of dimension $2N^{{2 \over 3}}$. Thus as we choose the $2N^{{2 \over 3}}$
elements of $\Xi_1$ and $\Xi_2$, the probability of introducing nullity at each selection is
bounded by $N^{-{4 \over 3} + f(\epsilon)}$. Thus the probability $p_k$ of having nullity $k$ is bounded by
$$p_k \lesssim  (N^{-{2 \over 3}+f(\epsilon)})^k.$$
Thus the expected density of $span(\Xi_1) \cap span(\Xi_2)$ is $O(N^{-{11 \over 3} + f(\epsilon)})$.
This contradicts \eqref{lowerbound}, so we have just proved that $\Delta$ cannot contain a set as thick in
$\Lambda + H$ as required by Theorem \ref{specstruct}, which is a contradiction.

Thus we have proved Theorem \ref{noincrements} and hence Theorem \ref{main}.

\section{Epilogue on efficiency}

The $\epsilon$ which we obtain in Theorem \ref{main} is of necessity quite small. Like many arguments in analysis,
ours does not aim for great efficiency and we seem to lose a factor in the size of $\epsilon$ in nearly
every line of the argument. However, there are a few points in the argument where we lose significantly more.
The most notable of these are the use of the asymmetric Balog Szemeredi Gowers lemma and the discovery
of the additive structure with comity in Corollary \ref{comityatlast}. What these two parts of the
argument have in common is that they are iterative. One starts with a structure and seeks a certain property. If the
property is lacking, one finds a structure which is in a certain sense better but which has cost us
a factor in $\epsilon$.

Iteration can be a powerful thing. It sometimes allows us to prove a very deep thing with very few words. The
proof doesn't care how many iterates must be calculated to implement it. We have a version of our argument
which we will publish elsewhere that adds even another layer of iteration to the argument but spares us
some of the difficulties of Theorem \ref{structuretheory}. 
But certainly, it is less efficient than the argument
presented here.

To improve efficiency, it is important to make the argument less iterative. In hopes of doing so, we leave the
reader with two conjectures:

\begin{conjecture} Let $N>0$ be a large parameter. (Constants may not depend on $N$.)
Suppose  $B$ and $C$ are sets in $F_3^N$ each of whose cardinality is bounded by $N^{100}$.
  Suppose the number of quadruplets $b_1+c_1=b_2+c_2$
with $b_1,b_2 \in B$ and $c_1,c_2 \in C$ is at least
$|B| |C|^2  N^{-\sigma}.$
Then $C$ is contained in a subspace of dimension $N^{O(\sigma)}$.  \end{conjecture}

\bigskip

\begin{conjecture} Suppose $\Delta$ is a subset of $F_3^N$ which is $N^{\sigma}$ additively non-smoothing.
(We may assume as in the previous conjecture that $|\Delta| \leq N^{100}$.
Then $\Delta$ admits an additive structure at some height with   $N^{O(\sigma)}$ comity.
\end{conjecture}

\bigskip
\bigskip

\tiny

\textsc{M. BATEMAN, DEPARTMENT OF MATHEMATICS, UCLA, LOS ANGELES CA}

{\it bateman@math.ucla.edu}

\bigskip

\textsc{N. KATZ, DEPARTMENT OF MATHEMATICS, INDIANA UNIVERSITY, BLOOMINGTON IN}

{\it nhkatz@indiana.edu}

\end{document}